\documentclass[10pt,reqno]{article}

\usepackage[b5paper,top=1.2in,left=0.9in]{geometry}
\usepackage{verbatim}
\usepackage{amssymb}
\usepackage{amsmath}
\usepackage{graphicx}
\usepackage{appendix}
\usepackage{color}
\usepackage{amsthm}
\usepackage{tikz}
\usepackage{ifthen}
\usepackage{float}
\usepackage{pdflscape}
\usepackage{rotating}
\usepackage{hyperref}
\usetikzlibrary{arrows,positioning,decorations.pathmorphing,  decorations.markings}

\renewcommand{\d}{\mathrm{d}}
\newcommand{\D}{\mathrm{D}}

\newcommand{\e}{\mathrm{e}}

\newtheorem{Thm}{Theorem}[section]
\newtheorem{Lem}[Thm]{Lemma}
\newtheorem{Prop}[Thm]{Proposition}
\newtheorem{Cor}[Thm]{Corollary}
\newtheorem{Con}[Thm]{Conjecture}
\newtheorem*{MainThm}{Main Theorem}

\theoremstyle{definition}
\newtheorem{Def}[Thm]{Definition}
\newtheorem{Rem}[Thm]{Remark}
\newtheorem{Nota}[Thm]{Notation}
\newtheorem{Ex}[Thm]{Example}

\newtheoremstyle{named}{}{}{\itshape}{}{\bfseries}{.}{.5em}{#1 #3}
\theoremstyle{named}

\def\R{\mathbb{R}}
\def\Q{\mathbb{Q}}

\def\C{\mathbb{C}}
\def\Z{\mathbb{Z}}

\def\fD{\mathfrak{D}}

\def\fb{\mathfrak{b}}
\def\sl{\mathfrak{sl}}
\def\g{\mathfrak{g}}
\def\gl{\mathfrak{gl}}
\def\fh{\mathfrak{h}}

\def\cA{\mathcal{A}}
\def\cB{\mathcal{B}}
\def\cC{\mathcal{C}}

\def\cF{\mathcal{F}}
\def\cH{\mathcal{H}}

\def\cK{\mathcal{K}}

\def\cP{\mathcal{P}}
\def\cR{\mathcal{R}}

\def\cU{\mathcal{U}}
\def\cV{\mathcal{V}}
\def\cW{\mathcal{W}}
\def\cX{\mathcal{X}}

\def\a{\alpha}
\def\b{\beta}
\def\e{\epsilon}

\def\c{\gamma}
\def\D{\Delta}

\def\d{\delta}

\def\k{\kappa}

\def\l{\lambda}
\def\L{\Lambda}

\def\w{\omega}

\def\bA{\mathbf{A}}

\def\bC{\mathbf{C}}
\def\bD{\mathbf{D}}

\def\bx{\mathbf{x}}

\def\be{\mathbf{e}}
\def\bE{\mathbf{E}}
\def\bf{\mathbf{f}}
\def\bF{\mathbf{F}}

\def\bH{\mathbf{H}}
\def\bi{\mathbf{i}}

\def\bJ{\mathbf{J}}

\def\bK{\mathbf{K}}

\def\bQ{\mathbf{Q}}

\def\bT{\mathbf{T}}

\def\=>{\Longrightarrow}
\def\inj{\hookrightarrow}
\def\corr{\longleftrightarrow}

\def\to{\longrightarrow}

\def\ox{\otimes}
\def\o+{\oplus}
\def\bo+{\bigoplus}
\def\x{\times}
\def\<{\langle}
\def\>{\rangle}
\def\oo{\infty}
\def\cong{\equiv}
\def\^{\wedge}
\def\+{\dagger}

\def\sub{\subset}
\def\inv{^{-1}}
\def\half{\frac12}
\def\dis{\displaystyle}
\def\dd[#1,#2]{\frac{d#1}{d#2}}
\def\del[#1,#2]{\frac{\partial #1}{\partial #2}}
\def\over[#1]{\overline{#1}}
\def\vec[#1]{\overrightarrow{#1}}
\def\:{\;:\;}

\def\tab{\;\;\;\;\;\;}

\newcommand{\til}[1]{\widetilde{#1}}
\newcommand{\what}[1]{\widehat{#1}}

\newcommand{\mat}[1]{\begin{pmatrix}#1\end{pmatrix}}

\newcommand{\case}[2][cccccccccccccccccccccccccccccccccccccccccc]{\left\{\begin{array}{#1}#2 \\ \end{array}\right.}
\newcommand{\Eq}[1]{\begin{align}#1\end{align}}
\newcommand{\Eqn}[1]{\begin{align*}#1\end{align*}}
\renewcommand{\over}[1]{\overline{#1}}

\tikzset{>=latex}
\tikzstyle{vthick}=[line width=1.8pt]

\newcommand\drawpath[2]{%
  \foreach \too [count=\c from 1] in {#1}
  {
  \ifthenelse{\c=1}
  {\xdef\from{\too}}
  {\path (\from) edge [->, #2] (\too);
    \xdef\from{\too}}
  };
}
\begin{document}
\title{Parabolic Positive Representations of $\cU_q(\g_\R)$}

\author{  Ivan C.H. Ip\footnote{
	  Department of Mathematics, Hong Kong University of Science and Technology\newline
	  Email: ivan.ip@ust.hk\newline		
 	 The author is supported by the Hong Kong RGC General Research Funds ECS \#26303319.
          }
}
\maketitle

\numberwithin{equation}{section}

\begin{abstract}
We construct a new family of irreducible representations of $\cU_q(\g_\R)$ and its modular double by quantizing the classical parabolic induction corresponding to arbitrary parabolic subgroups, such that the generators of $\cU_q(\g_\R)$ act by positive self-adjoint operators on a Hilbert space. This generalizes the well-established positive representations which corresponds to induction by the minimal parabolic (i.e. Borel) subgroup. We also study in detail the special case of type $A_n$ acting on $L^2(\R^n)$ with minimal functional dimension, and establish the properties of its central characters and universal $\cR$ operator. We construct a positive version of the evaluation modules of the affine quantum group $\cU_q(\what{\sl}_{n+1})$ modeled over this minimal positive representation of type $A_n$.
\end{abstract}
\vspace{3mm}

\vspace{3mm}


\tableofcontents
\section{Introduction}\label{sec:intro}

\subsection*{Positive representations of $\cU_q(\g_\R)$}
\emph{Positive representations} was introduced in \cite{FI} to study the representation theory of \emph{split real quantum groups} $\cU_{q}(\g_\R)$ associated to semisimple Lie algebra $\g$, as well as its \emph{modular double} $\cU_{qq^\vee}(\g_\R)$ introduced by \cite{Fa1, Fa2} in the regime where $|q|=1$. These representations are natural generalizations of a special class of representations of $\cU_q(\sl(2,\R))$ originally studied by Teschner \emph{et al.} \cite{BT, PT1, PT2} from the physics point of view of quantum Liouville theory, which is characterized by the actions of \emph{positive (essentally) self-adjoint operators} on the Hilbert space $L^2(\R)$. 

Based on quantizing the regular action on smooth functions on the totally positive flag variety $G_{>0}/B_{>0}$, we constructed, for the simply-laced cases in \cite{FI, Ip2} and non-simply-laced cases in \cite{Ip3}, a family of irreducible representations $\cP_\l$ of $\cU_{qq^\vee}(\g_\R)$ with the generators acting on certain Hilbert space as positive self-adjoint operators, parametrized by $\l\in P_{\R^+}$ in the positive real-span of the dominant weights. 

Recently a cluster algebraic realization of these representations were also constructed, first for type $A_n$ in \cite{SS1} and the general case in \cite{Ip7}, where we established an embedding of $\cU_q(\g)$ into certain quantum torus algebra $\cX_q$ associated to the \emph{basic quiver} $\bD(\bi_0)$, such that a choice of \emph{polarization} of $\cX_q$ coincides with the positive representations. The positive representations and their cluster realizations are long known to be closely related to quantum higher Teichm\"uller theory \cite{FG1, FG3, Ka1}, and recently a full geometric interpretation based on the modular space of decorated $G$-local system is given in the monumental work of Goncharov and Shen \cite{GS}.

\subsection*{Main results}
In this paper, we construct a large class of irreducible representations of $\cU_{qq^\vee}(\g_\R)$, called the \emph{parabolic positive representations}, by quantizing the regular action on the totally positive part of the partial flag variety $G/P$ for \emph{any} parabolic subgroups. These representations are still characterized by the fact that the quantum group generators act by positive self-adjoint operators on some Hilbert spaces. In particular, the original positive representation is a special case of this new family of representations. To summarize, the main results of this paper are the following (Theorem \ref{main}, Corollary \ref{maincor}, Theorem \ref{mainmod}).

\begin{MainThm} We have a homomorphism of the Drinfeld's double $\fD_q(\g)$ onto a quantum torus algebra $\cX_q^{\bD(\bi)}$ associated to a subquiver $\bD(\bi)$ of the basic quiver $\bD(\bi_0)$ of the positive representations, such that a polarization of $\cX_q^{\bD(\bi)}$ provides a positive representation $\cP_\l^J$ for $\cU_{q}(\g_\R)$ and its modular double $\cU_{qq^\vee}(\g_\R)$, and it is a twisted quantization of the parabolic induction.
\end{MainThm}

The proof of the Main Theorem consists of several ingredients. First we review the construction of the basic quivers $\bD(\bi_0)$ in Section \ref{sec:pos:basicq}, which were first constructed in \cite{Ip7} for a reduced word $\bi_0$ of the longest element $w_0$ of the Weyl group. In this paper, we adapt the more general construction due to \cite{GS} in order to give an explicit realization of the \emph{extra vertices}, which were artificially added to the quiver previously in \cite{Ip7}. Next, we define a new notion of \emph{generalized Heisenberg double} (Definition \ref{defgenH}), which decomposes the image of $\fD_q(\g)$ in the cluster algebra $\cX_q$ satisfying a new set of algebraic relations, and finally we prove the Decomposition Lemma (Lemma \ref{decomp}), which decomposes the original positive representations in order to give us the parabolic ones. This requires an understanding of the combinatorics of the Coxeter moves of a reduced word (Lemma \ref{coxeter}), which was previously known to the author but never explicitly written down. We remark in Section \ref{sec:ppr:mod} that the construction naturally works for the modular double $\cU_{qq^\vee}(\g_\R)$ as well, thus establishing the Main Theorem.

\subsection*{Minimal positive representations}
A special case of the parabolic positive representations is worthy of special attention, which we call the \emph{minimal positive representations}. These are the representations of $\cU_q(\sl(n+1,\R))$ in type $A_n$, where we take the parabolic subgroup $P$ to be the maximal one, such that its codimension in $SL(n+1,\R)$ is the minimum. Corresponding to this choice, the quantum group $\cU_q(\g_\R)$ acts on the minimal positive representation $\cP_\l^{\min}\simeq L^2(\R^n)$ as positive operators parametrized by a single number $\l\in\R$, and this Hilbert space has the smallest functional dimension possible.

In fact, due to its simplicity, these representations served as our original motivation to investigate a possibility of degenerating the well-established positive representations, and we later realized a way to extend the construction to arbitrary parabolic cases. The simplicity of such representations is expected to find applications in mathematical physics and integrable systems. Therefore we study this family in more detail and establish several important properties concerning the universal $\cR$ operator and the central characters (Theorem \ref{minR}, Theorem \ref{minCas}).
\begin{MainThm} Acting on the minimal positive representations,
\begin{itemize}
\item The universal $\cR$ operator is well-defined as a unitary transformation on the tensor product $\cP_\l^{\min}\ox\cP_{\l'}^{\min}$ of minimal positive representations.
\item The Casimir operators acts by real scalars on $\cP_\l^{\min}$, and lies outside the spectrum of the positive Casimirs.
\end{itemize}
\end{MainThm}

Finally, the simplicity of the minimal positive representations, interpreted as a homomorphic image of the Drinfeld's double $\fD_q(\g)$ onto a quantum torus algebra $\cX_q^\bQ$ associated to a very simple quiver $\bQ$, allows us to construct a new class of positive representations for the quantum affine algebra $\cU_q(\what{\sl}(n+1,\R))$ (although it is no longer parabolic). It is still an open question to appropriately define the representation theory of split real affine quantum groups $\cU_q(\what{\g}_\R)$, and this accidental construction may allow for a first step. We show that the representations we obtained is isomorphic to the \emph{evaluation module} of Jimbo \cite{Ji2} associated to the minimal positive representation of $\cU_q(\sl(n+1,\R))$ constructed here.

\subsection*{Geometric interpretation} 
In this paper, we follow our previous strategy in \cite{FI, Ip2, Ip7} to establish a positive representations modeled over the totally positive part of the partial flag variety $G/P$, and study its quantization through algebraic means, establishing a homomorphism of the Drinfeld's double onto the quantum torus algebra $\cX_q^{\bD(\bi)}$ associated to the (double of the) basic quiver $\bQ(\bi)$. However, there are new difficulties that are not present in the usual positive representation case. For example, in the general parabolic case, we do not have a natural association of these basic quivers to triangles of the triangulation of punctured Riemann surfaces due to the degeneracy of the frozen degrees on the edges. Furthermore, several proofs in \cite{FI, Ip2, Ip3} and \cite{GS} utilized the freedom of the longest reduced word $\bi_0$ (more precisely, the existence of any reduced word of $w_0$ starting with any letter) which is not available in the parabolic case as well. 

According to \cite{GS}, the basic quiver $\bQ(\bi)$ is naturally associated to the Poisson structure of the \emph{partial configuration space} $\mathrm{Conf}_u^e(\cA)$ of the principal affine space $\cA$ for some element $u\in W$ of the Weyl group. Therefore one should try to understand and possibly simplify the proofs of the parabolic positive representations by quantizing the geometrical methods using, perhaps, partial decorated $G$-local system, where the decorations are provided by partial flags, and study the combinatorics of the its potential functions $\cW_i$. In particular, we expect that the Lusztig's braid group action can be established through the geometric point view of the parabolic positive representations.

\subsection*{Outline} The outline of the paper is as follows. In Section \ref{sec:pre}, we establish our notations and give preliminaries to the notion of total positivity, parabolic subgroups, quantum groups, Drinfeld's double, and quantum torus algebra. In Section \ref{sec:pos} we summarize all the structural results of positive representations and its cluster realization, adapting to the notations used in this paper. In particular we recall the construction of the basic quiver. In Section \ref{sec:min}, we give our motivation with the toy model of $\cU_q(\sl(4,\R))$, and construct the minimal positive representations for $\cU_q(\sl(n+1,\R))$ in general, establishing some of its properties. We also generalize this construction to the case of affine quantum group $\cU_q(\what{\sl}(n+1,\R))$ and compare with its evaluation modules. In Section \ref{sec:ppr}, we construct the parabolic positive representation in full generality, by proving the Decomposition Lemma of the quantum group generators. In Section \ref{sec:ex} we provide two examples to demonstrate the construction of this paper, and finally in Section \ref{sec:open} we discuss several open research directions.

\section{Preliminaries}\label{sec:pre}
\subsection{Total positivity and parabolic subgroups}

Let $\g$ be a simple\footnote{The results of this paper obviously also work for the semisimple case by taking direct product.} Lie algebra over $\C$. Let $G$ be the real simple Lie group corresponding to the split real form $\g_\R$ of the Lie algebra $\g$, and let $B, B^-$ be two opposite Borel subgroups containing a split real maximal torus $T=B\cap B^-$. Let $U^+\subset B, U^-\subset B^-$ be the corresponding unipotent subgroups.

Let $I$ be the root index of the Dynkin diagram of $\g$ such that
\Eq{
|I|=n=\mathrm{rank}(\g).}
Let $\Phi$ be the set of roots of $\g$, $\Phi_+\subset \Phi$ be the positive roots, and $\D_+=\{\a_i\}_{i\in I}\subset \Phi_+$ be the positive simple roots. Let $W=\<s_i\>_{i\in I}$ be the Weyl group of $\Phi$ generated by the simple reflections $s_i:=s_{\a_i}$.
\begin{Def} Let $(-,-)$ be a $W$-invariant inner product of the root lattice. We define
\Eq{
a_{ij}:=\frac{2(\a_i,\a_j)}{(\a_i,\a_i)},\tab i,j\in I
}
such that $A:=(a_{ij})$ is the \emph{Cartan matrix}. 

We normalize $(-,-)$ as follows: we choose the symmetrization factors (also called the \emph{multipliers}) such that for any $i\in I$,
\Eq{\label{di}d_i:=\frac{1}{2}(\a_i,\a_i)=\case{1&\mbox{$i$ is long root or in the simply-laced case,}\\\frac{1}{2}&\mbox{$i$ is short root in type $B,C,F$,}\\\frac{1}{3}&\mbox{$i$ is short root in type $G$,}}}
and $(\a_i,\a_j)=-1$ when $i,j \in I$ are adjacent in the Dynkin diagram, such that
\Eq{
d_ia_{ij}=d_ja_{ji}.
}
\end{Def}
\begin{Def}
Let $w\in W$. We call a sequence 
\Eq{\bi=(i_1,...,i_k),\tab i_k\in I} a \emph{reduced word} of $w$ if 
$$w=s_{i_1}s_{i_2}\cdots s_{i_k}$$
is a reduced expression, and let 
\Eq{
l(w):=|\bi|:=k
} be its length. We write 
\Eq{
\bi^{op}:=(i_k,...,i_1)} for the reduced word of $w\inv$.

Let $w_0$ be the longest element of $W$. Throughout the paper we let 
\Eq{
N:=l(w_0)}
and denote by $\bi_0$ a reduced word of $w_0$.
\end{Def}
A useful fact is the following:
\begin{Prop}\label{w0head}\cite[Lemma 10.8]{GS} If $\bi=(i_1,...,i_k)$ is a reduced word, then there exists a reduced word $\bi_0$ of $w_0$ starting with $\bi$.
\end{Prop}

\begin{Def} Let $w_0\in W$ be the longest element of the Weyl group. The Dynkin involution
\Eq{
I&\to I\nonumber\\
i&\mapsto i^*
}
is defined by
\Eq{w_0 s_i w_0 = s_{i^*}.}
Equivalently, we have
\Eq{
w_0(\a_i) = \a_{-i^*},\tab \a_i\in \D_+
}
where $\a_{-i}$ are the negative simple roots.

\end{Def}

Next we recall the description of the Lusztig's data for total positivity, given in detail in \cite{Lu}. 
\begin{Def}For any $i\in I$, there exists a homomorphism $$SL_2(\R)\to G$$ induced by \footnote{We used $h_i$ instead of the more traditional $\chi_i$ to distinguish it from the characters $\chi_\l$ used later.}
\begin{eqnarray}
\mat{1&a\\0&1}&\mapsto& x_i(a)\in U_i^+,\\
\mat{b&0\\0&b\inv}&\mapsto &h_i(b)\in T,\\
\mat{1&0\\c&1}&\mapsto &y_i(c)\in U_i^-,
\end{eqnarray}
called a \emph{pinning} of $G$, where $U_i^+\subset U^+$ and $U_i^-\subset U^-$ are the \emph{simple root subgroups} of the unipotent subgroup $U^+$ and $U^-$ respectively. 
\end{Def}

\begin{Lem}\cite[Proposition 2.7]{Lu}\label{inj} 
Let $\bi_0=(i_1,...,i_N)$ be a reduced word of the longest element $w_0$. The map 
\Eq{
\iota:\R_{>0}^N&\to U^+\nonumber\\
\label{w0coord}\iota:(a_1,a_2,...,a_N)&\mapsto x_{i_1}(a_1)x_{i_2}(a_2)...x_{i_N}(a_N)
}
is injective. The positive unipotent semigroup $U_{>0}^+$ is defined to be the image of $\iota$:
\Eq{
U_{>0}^+:=\iota(\R_{>0}^N).
}
We have the similar definition for $U_{>0}^-$ using $y_i$ instead.
\end{Lem}

\begin{Def} The totally positive semigroup is defined to be
\Eq{\label{gauss}
G_{>0}:=U_{>0}^- T_{>0} U_{>0}^+,
} where $U_{>0}^\pm$ is as above, and $T_{>0}$ are generated by the images $h_i(b)$ with $b\in\R_{>0}$, $i\in I$.
\end{Def}
\begin{Lem} We have the following identities in $G_{>0}$: for $a,b,c\in \R_{>0}$ and $i,j\in I$,
\Eq{
\label{EF}x_i(a)y_j(c)&=y_j(c)x_i(a),&& \mbox{ if $i\neq j$},\\
\label{EK}h_i(b)x_j(a)&=x_j(b^{a_{ij}} a)h_i(b),\\
\label{EKF}x_i(a)h_i(b)y_i(c)&=y_i(\frac{c}{ac+b^2})h_i(\frac{ac+b^2}{b})x_i(\frac{a}{ac+b^2}),\\
\label{EE0}x_i(a)x_j(b)&=x_j(b)x_i(a),&& a_{ij}=0,\\
\label{EE1}x_i(a)x_j(b)x_i(c)&=x_j(\frac{bc}{a+c})x_i(a+c)x_j(\frac{ab}{a+c}), && a_{ij}=-1,\\
\label{EE2}x_i(a)x_j(b)x_i(c)x_j(d)&=x_j(\frac{bc^2d}{S})x_i(\frac{S}{R})x_j(\frac{R^2}{S})x_i(\frac{abc}{R}),&& a_{ij}=-2,
}
where
\Eq{
R=ab+ad+cd\tab S=a^2b+d(a+c)^2.}
We also have an explicit expression for the case $a_{ij}=-3$, see \cite[Theorem 3.1]{BZ}, but we will not need them in this paper. 
The same relations \eqref{EE0}--\eqref{EE2} also hold for $y_i$.
\end{Lem}

Next we recall some terminologies regarding parabolic subgroups that are essential in this paper.
\begin{Def}
Let $J\subset I$ be a (possibly empty) subset of the Dynkin nodes, and $W_J\subset W$ the corresponding subgroups generated by $s_j, j\in J$. The \emph{Levi subgroup} $L_J$ is the subgroup of $G$ generated by the maximal torus $T$ and the root subgroups $U_j^+, U_j^-$ where $j\in J$. The \emph{standard parabolic subgroup}\footnote{
We used the opposite version instead of the usual $P_J=L_JB$ to fit the calculations later, so that in $SL_n$ the parabolic subgroups $P_J$ are represented by  lower block triangular matrices.} $P_J$ is defined as $P_J=B_-L_J$.

We write 
\Eq{
N_J:=l(w_J)
} to be the length of the \emph{longest element} $w_J\in W_J$ of the Weyl subgroup $W_J\subset W$. Hence the codimension of $P_J$ in $G$ is given by
\Eq{
\mathrm{codim}_G(P_J)=N-N_J.
}
\end{Def}
It is well-known that any parabolic subgroup, i.e. subgroup containing a Borel subgroup, is conjugated to $P_J$ for some $J$.
\begin{Rem} By convention, if $J=\emptyset$, then $P_J:=B_-$ and $W_J=\{e\}$ is the trivial subgroup.
\end{Rem}

\begin{Ex} For $G=SL_5(\R)$, let $J=\{1,2,4\}$. Then the standard parabolic subgroup $P_J\subset G$ is the set of $5\x 5$ matrix of the form
\Eq{
P_J=\mat{*&*&*&0&0\\{*}&*&*&0&0\\{*}&*&*&0&0\\{*}&*&*&*&*\\{*}&*&*&*&*}.
}
\end{Ex}

If $P_J$ is a parabolic subgroup, we denote by 
\Eq{
P_{J}^{>0}:=G_{>0}\cap P
} its totally positive part. 

It follows from \eqref{gauss} that we have the \emph{Langlands decomposition} as well in the totally positive case.
\begin{Lem} Let $w_J$ be the longest element of the subgroup $W_J\subset W$ and $\bi_J=(j_1,...,j_{N_J})$, $j_k\in J\subset I$ be its reduced expression. Then we have
\Eq{
P_J^{>0}=U_{>0}^- T_{>0} M_{>0}\label{langlandsd}
}
where $M_{>0}$ is the image of the embedding
\Eq{\iota: \R_{>0}^{N_J} &\to U_{>0}^+ \nonumber\\
(a_1,...,a_{N_J}) &\mapsto x_{j_1}(a_1)x_{j_2}(a_2)\cdots x_{j_{N_J}}(a_{N_J}).\label{paraLuscoord}
}
\end{Lem}
Since $P_J$ contains the full subgroup of $SL_2$ under the pinning corresponding to the simple root $\a_j$ where $j\in J$, we have the following characterization of the \emph{characters} of $P_J^{>0}$ in the totally positive scenario:
\begin{Lem}\label{parachar} Let 
\Eq{
h_1(a_1)\cdots h_n(a_n)\in h_{>0},\tab a_i\in \R_{>0}}
be the coordinates of the abelian component $T_{>0}$ of the Langlands decomposition of $g\in P_J^{>0}$. Then any positive character $$\chi:P_J^{>0}\to \R_{>0}$$ of the totally positive part $P_J^{>0}$ is parametrized by the scalars $\l=(\l_k)\in \R^{|I\setminus J|}$, given by
\Eq{
\chi_\l(g) := \prod_{k\in I\setminus J} a_k^{\l_k},\tab g\in P_J^{>0}.
}
\end{Lem}

\subsection{Quantum groups $\cU_q(\g)$ and $\fD_q(\g)$}\label{sec:pre:nota}
For any finite dimensional complex semisimple Lie algebra $\g$, Drinfeld \cite{Dr} and Jimbo \cite{Ji1} associated to it a remarkable Hopf algebra $\cU_q(\g)$ known as \emph{quantum group}, which is certain deformation of the universal enveloping algebra. We follow the notations used in \cite{Ip7} for $\cU_q(\g)$ as well as the Drinfeld's double $\fD_q(\g)$ of its Borel part.

In the following, we assume again that $\g$ is of simple Dynkin type.
\begin{Def} \label{qi} Let $d_i$ be the multipliers \eqref{di}. We define
\Eq{
q_i:=q^{d_i},
} which we will also write as
\Eq{
q_l&:=q,\\ 
q_s&:=\case{q^{\frac12}&\mbox{if $\g$ is of type $B, C, F$},\\q^{\frac13}&\mbox{if $\g$ is of type $G$},}
}
for the $q$ parameters corresponding to long and short roots respectively.
\end{Def}

\begin{Def} We define $\fD_q(\g)$ to be the $\C(q_s)$-algebra generated by the elements
$$\{\bE_i, \bF_i,\bK_i^{\pm1}, \bK_i'^{\pm1}\}_{i\in I}$$ 
subject to the following relations (we will omit the obvious relations involving $\bK_i^{-1}$ and ${\bK_i'}^{-1}$ below for simplicity):
\Eq{
\bK_i\bE_j&=q_i^{a_{ij}}\bE_j\bK_i, &\bK_i\bF_j&=q_i^{-a_{ij}}\bF_j\bK_i,\label{KK1}\\
\bK_i'\bE_j&=q_i^{-a_{ij}}\bE_j\bK_i', &\bK_i'\bF_j&=q_i^{a_{ij}}\bF_j\bK_i',\label{KK2}\\
\bK_i\bK_j&=\bK_j\bK_i, &\bK_i'\bK_j'&=\bK_j'\bK_i', &\bK_i\bK_j' = \bK_j'\bK_i,\label{KK3}\\
&&[\bE_i,\bF_j]&= \d_{ij}\frac{\bK_i-\bK_i'}{q_i-q_i\inv},\label{EFFE}
}
together with the \emph{Serre relations} for $i\neq j$:
\Eq{
\sum_{k=0}^{1-a_{ij}}(-1)^k\frac{[1-a_{ij}]_{q_i}!}{[1-a_{ij}-k]_{q_i}![k]_{q_i}!}\bE_i^{k}\bE_j\bE_i^{1-a_{ij}-k}&=0,\label{SerreE}\\
\sum_{k=0}^{1-a_{ij}}(-1)^k\frac{[1-a_{ij}]_{q_i}!}{[1-a_{ij}-k]_{q_i}![k]_{q_i}!}\bF_i^{k}\bF_j\bF_i^{1-a_{ij}-k}&=0,\label{SerreF}
}
where $\dis [k]_q:=\frac{q^k-q^{-k}}{q-q\inv}$ is the $q$-number, and $\dis [n]_q!:=\prod_{k=1}^n [k]_q$ is the $q$-factorial.
\end{Def}

The algebra $\fD_q(\g)$ is a Hopf algebra with comultiplication
\Eq{
\D(\bE_i)&=1\ox \bE_i+\bE_i\ox \bK_i,&\D(\bK_i)&=\bK_i\ox \bK_i,\\
\D(\bF_i)&=\bF_i\ox 1+\bK_i'\ox \bF_i,&\D(\bK_i')&=\bK_i'\ox \bK_i',
}
We will not need the counit and antipode in this paper.

\begin{Def}
The quantum group $\cU_q(\g)$ is defined as the quotient
\Eq{
\cU_g(\g):=\fD_q(\g)/\<\bK_i\bK_i'=1\>_{i\in I},
}
and it inherits a well-defined Hopf algebra structure from $\fD_q(\g)$. 
\end{Def}
\begin{Rem} $\fD_q(\g)$ is the \emph{Drinfeld's double} of the quantum Borel subalgebra $\cU_q(\fb)$ generated by $\bE_i$ and $\bK_i$.
\end{Rem}
\begin{Def}We define the rescaled generators
\Eq{
\be_i &:=\left(\frac{\sqrt{-1}}{q_i-q_i\inv}\right)\inv \bE_i,&\bf_i &:=\left(\frac{\sqrt{-1}}{q_i-q_i\inv}\right)\inv \bF_i.\label{rescaleFF}
}
By abuse of notation\footnote{This is the convention adopted in e.g. \cite{GS, SS2}.}, we will also denote by $\fD_q(\g)$ the $\C(q_s)$-algebra generated by 
\Eq{\{\be_i, \bf_i, \bK_i, \bK_i'\}_{i\in I}
} and the corresponding quotient by $\cU_q(\g)$. The generators satisfy all the defining relations above except \eqref{EFFE} which is modified to be 
\Eq{\label{EFFE2}
[\be_i, \bf_j]=\d_{ij} (q_i-q_i\inv)(\bK_i'-\bK_i).
}
\end{Def}

In the split real case, we require $|q|=1$ and write
\Eq{q:=e^{\pi \sqrt{-1} b^2}} 
where $0<b<1$.
\begin{Def}
We define $\cU_q(\g_\R)$ to be the real form of $\cU_q(\g)$ induced by the star structure
\Eq{\be_i^*=\be_i,\tab \bf_i^*=\bf_i,\tab \bK_i^*=\bK_i,}
with $q^*=\over{q}=q\inv$, making it a Hopf-* algebra.
\end{Def}

Finally, we recall the Lusztig's transformations, which gives a braid group action $\cB$ on $\cU_q(\g)$ as automorphisms. Here we introduce the positive version constructed in \cite{Ip4}. We only need the simply-laced case in this paper.
\begin{Def} \label{lusztigbraid}
In the simply-laced case, the \emph{positive Lusztig's braid group action} is given by
$$
T_i: \cU_q(\g)\to \cU_q(\g),\tab i\in I,$$
such that
\Eq{
T_i(\bK_j)&:=\bK_j\bK_i^{-a_{ij}},\\
T_i(\be_i)&:=q\bK_i\bf_i,\\
T_i(\be_j)&:=\be_j,& a_{ij}&=0,\\
T_i(\be_j)&:=\frac{[\be_j,\be_i]_{q^{1/2}}}{q-q\inv},& a_{ij}&=-1,
}
where the $q$-commutator is defined as
\Eq{
[X,Y]_q:=qXY-q\inv YX,
}
and similarly for the $\bf_i$ generators, such that the image is still self-adjoint.
\end{Def}
A well known fact is that the Lusztig's braid group action induces the interchange of generators as follows:
\begin{Prop}\label{Ti0} If $\bi_0=(i_1,...,i_N)$ is a reduced word for the longest element $w_0\in W$, then the automorphism
\Eq{T_{\bi_0}:=T_{i_1}\cdots T_{i_N}}
is given by
\Eqn{
T_{\bi_0}(\be_i) &= \bf_i,\\
T_{\bi_0}(\bf_i) &= \be_i,\\
T_{\bi_0}(\bK_i) &= \bK_i\inv.\\
}
\end{Prop}
\begin{Ex}\label{nonsimple}
In the special case when $\g=\sl_{n+1}$, we define for $1\leq i<j\leq n$,
\Eqn{
\be_{i,j}&:=T_{i}T_{i+1}\cdots T_{j-1}(\be_{j}),\\
\bf_{i,j}&:=T_{i}T_{i+1}\cdots T_{j-1}(\bf_{j}).
}
The index (shifted by one appropriately) coincides with the corresponding entries in the $(n+1)$-dimensional fundamental representation of $\cU_q(\sl_{n+1})$.  Under the natural order of positive roots 
\Eq{\c_k:= s_{i_1}...s_{i_{k-1}}\a_{i_k} \in \Phi_+,\tab k=1,...,N}
corresponding to the standard reduced word
\Eq{\label{standardAw0}
\bi_0=(1,2,3,...,n,1,2,...,n-1,...,1,2,1)
} of the longest element $w_0$, the (ordered) monomials of these generators form the \emph{positive PBW basis} of $\cU_q(\sl_{n+1})$ \cite{Ip4}.
\end{Ex}
\subsection{Quantum torus algebra}\label{sec:pre:qtorus}
In this subsection we recall some definitions and properties concerning quantum torus algebra and their cluster realizations.
\begin{Def} A \emph{cluster seed} is a datum 
\Eq{
\bQ=(Q, Q_0, B, D)
} where $Q$ is a finite set, $Q_0\subset Q$ is a subset called the \emph{frozen subset}, $B=(b_{ij})_{i,j\in Q}$ a skew-symmetrizable $\frac12\Z$-valued matrix called the \emph{exchange matrix}, and $D=diag(d_j)_{j\in Q}$ is a diagonal $\Q$-matrix called the \emph{multiplier} with $d_j\inv\in \Z$, such that
\Eq{W:=DB=-B^TD} is skew-symmetric. 
\end{Def}
In the following, we will consider only the case where there exists a \emph{decoration} 
\Eq{\eta:Q\to I
} to the root index of a simple Dynkin diagram, such that $D=\mathrm{diag}(d_{\eta(j)})_{j\in Q}$ where $(d_i)_{i\in I}$ are the multipliers given in \eqref{di}.

Let $\L_\bQ$ be a $\Z$-lattice with basis $\{\vec[e_i]\}_{i\in Q}$, and let $d=\min_{j\in Q}(d_j)$. Also let
\Eq{\label{wij}w_{ij}=d_ib_{ij}=-w_{ij}.}
We define a skew symmetric $d\Z$-valued form $(-,-)$ on $\L_\bQ$ by 
\Eq{
(\vec[e_i], \vec[e_j]):=w_{ij}.
}

\begin{Def} Let $q$ be a formal parameter. We define the \emph{quantum torus algebra} $\cX_q^{\bQ}$ associated to a cluster seed $\bQ$ to be an associative algebra over $\C[q^d]$ generated by $\{X_i^{\pm 1}\}_{i\in Q}$ subject to the relations
\Eq{\label{XiXj}
X_iX_j=q^{-2w_{ij}}X_jX_i,\tab i,j\in Q.
}
The generators $X_i\in \cX_q^{\bQ}$ are called the \emph{quantum cluster variables}, and they are \emph{frozen} if $i\in Q_0$. We denote by $\bT_q^\bQ$ the non-commutative field of fractions of $\cX_q^\bQ$.

Alternatively, $\cX_q^\bQ$ is generated by $\{X_{\l}\}_{\l\in \L_\bQ}$ with $X_0:=1$ subject to the relations
\Eq{q^{(\l,\mu)}X_{\l}X_{\mu} = X_{\l+\mu},\tab \mu,\l\in\L_\bQ.}
\end{Def}
\begin{Nota}\label{xik}
Under this realization, we shall write 
\Eq{
X_i=X_{\vec[e_i]},
} and define the notation
\Eq{X_{i_1,...,i_k}:=X_{\vec[e_{i_1}]+...+\vec[e_{i_k}]},}
or more generally for $n_1,...,n_k\in \Z$,
\Eq{X_{i_1^{n_1},...,i_k^{n_k}}:=X_{n_1\vec[e_{i_1}]+...+n_k\vec[e_{i_k}]}.}
\end{Nota}
\begin{Def}\label{quiver} 
We associate to each cluster seed $\bQ=(Q,Q_0,B,D)$ with decoration $\eta$ a quiver, denoted again by $\bQ$, with vertices labeled by $Q$ and adjacency matrix $C=(c_{ij})_{i,j\in Q}$, where
\Eq{
c_{ij}=\case{b_{ij}&\mbox{if $d_i=d_j,$}\\w_{ij}&\mbox{otherwise.}}
}
We call $i\in Q$ a \emph{short} (resp. \emph{long}) \emph{node} if $q_i:=q^{d_i}=q_s$ (resp. $q_i=q_l=q$). An arrow $i\to j$ represents the algebraic relation 
\Eq{
X_iX_j=q_*^{-2}X_jX_i} where $q_*=q_s$ if both $i,j$ are short nodes, or $q_*=q$ otherwise.
\end{Def}
Obviously one can recover the cluster seed from the quiver and the multipliers.
\begin{Nota}
We will use squares to denote frozen nodes $i\in Q_0$ and circles otherwise. We will also use dashed arrows if $c_{ij}=\half$, which only occur between frozen nodes. 

We will represent the algebraic relations \eqref{XiXj} by thick or thin arrows (see Figure \ref{thick}) for display conveniences (thickness is \emph{not} part of the data of the quiver). Thin arrows only occur in the non-simply-laced case between two short nodes.

Finally, we may omit the superscript $\bQ$ in $\cX_q^{\bQ}$ or $\bT_q^{\bQ}$ if the datum / quiver is clear from the context.
\end{Nota}
   
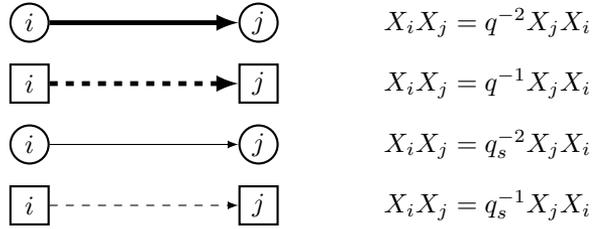
\begin{figure}[H]
\centering
  \begin{tikzpicture}[every node/.style={inner sep=0, minimum size=0.5cm, thick}, x=1cm, y=0.8cm]
\node[draw,circle] (i1) at(0,4) {$i$};
  \node[draw,circle] (j1) at (3,4) {$j$};
   \draw[vthick, ->](i1) to (j1);
\node at (6,4) {$X_iX_j = q^{-2} X_j X_i$}; 
\node[draw] (i2) at(0,3) {$i$};
  \node[draw] (j2) at (3,3) {$j$};
   \draw[vthick, dashed, ->](i2) to (j2);
\node at (6,3) {$X_iX_j = q^{-1} X_j X_i$};
\node[draw,circle] (i3) at(0,2) {$i$};
  \node[draw,circle] (j3) at (3,2) {$j$};
   \draw[thin, ->](i3) to (j3);
\node at (6,2) {$X_iX_j = q_s^{-2} X_j X_i$};
\node[draw] (i4) at(0,1) {$i$};
  \node[draw] (j4) at (3,1) {$j$};
   \draw[thin,dashed, ->](i4) to (j4);
\node at (6,1) {$X_iX_j = q_s^{-1} X_j X_i$};
  \end{tikzpicture}
  \caption{Arrows between nodes and their algebraic meaning.}\label{thick}
  \end{figure}

\begin{Def}\label{posrepX} A \emph{polarization} $\pi$ of the quantum torus algebra $\cX_q^\bQ$ on a Hilbert space $\cH=L^2(\R^M)$ is an assignment
\Eq{
X_i\mapsto e^{2\pi b L_i},\tab i\in Q,
}
where $L_i:=L_i(u_k, p_k, \l_k)$ is a linear combination of the position and momentum operators  $\{u_k, p_k\}_{k=1}^M$ satisfying the Heisenberg relations
\Eq{
[u_j, p_k]:=\frac{\d_{jk}}{2\pi\sqrt{-1}}
}
and complex parameters $\l_k \in \R$, such that they satisfy algebraically
\Eq{[L_i, L_j]=\frac{w_{ij}}{2\pi\sqrt{-1}}.} 
Each generator $X_i$ acts as a positive essentially self-adjoint operator on $\cH$ and gives a representation of $\cX_q^\bQ$ on $\cH$.
\end{Def}
\begin{Rem}\label{domain} The domains of these positive unbounded operators are taken to be the largest subspace $\cW$ of the Schwartz space of $L^2(\R^M)$ invariant under the exponential actions $e^{2\pi bL_i}$. They are essentially self-adjoint over the dense subspace 
\Eq{\{e^{-\bx^T \bA \bx+\b\cdot \bx}P(\bx)\;| \; \bA\mbox{ positive definite}, \b\in\C^M, P(\bx)\mbox{ polynomial}\}\subset \cW.}
These spaces are discussed in detail in e.g. \cite{FG3, Ip1, PT2} and in this paper we assume the algebraic relations among these operators are well-defined on $\cW$.
\end{Rem}

\begin{Nota}\label{ELnote}We will simplify notations and write
\Eq{\label{EL} 
e(L):=e^{\pi b L}
}
for $L$ a linear combination of position, momentum operators and scalars as above. Observe that we have
\Eq{
e(X)e(Y)=q^{\frac12} e(Y)e(X)
}
on $\cW$ whenever $[X,Y]=\frac{1}{2\pi\sqrt{-1}}$.

We also denote by
\Eq{\label{ELL}
[L_1]e(L_2) := e(L_1+L_2)+e(-L_1+L_2).
}
\end{Nota}
It is straightforward to check that
\begin{Prop}\label{standpolar} Let $\bQ$ be a cluster seed and $\cX_q^{\bQ}$ the quantum torus algebra. Then the following assignment 
$$X_i := e(-2p_i+\sum_{j\in Q} w_{ij} u_j)$$
acting on $L^2(\R^{|Q|})$ is a polarization. 
\end{Prop}
We call the assignment in Proposition \ref{standpolar} the \emph{standard polarization} of $\cX_q^{\bQ}$. Note that the position and momentum parts commute, so the action on the dense subspace $\cW$ can be explicitly written down as
\Eq{
X_i\cdot f(..., u_i, ...) = e^{\pi b\sum w_{ij}u_j}f(..., u_i+\sqrt{-1}b, ...)
}
where the shift is in the complex direction.
\begin{Lem}\label{polaru} Assume the rank of the skew-symmetric form of the lattice $\L_\bQ$ is $2M$. Then there is a polarization of $\cX_q^\bQ$ on $\cH=L^2(\R^M)$ and any polarization of $\cX_q^{\bQ}$ on $\cH$ is unitary equivalent by an $Sp(2M)$ action on the lattice $\L_\bQ$ (known as the Weil representation \cite{GS}).
\end{Lem}

Next we recall the notion of quantum cluster mutations.
\begin{Def}\label{qmut} Given a cluster seed $\bQ=(Q,Q_0,B, D)$ and an element $k\in Q\setminus Q_0$, a \emph{cluster mutation in direction $k$} is another seed $\bQ'=(Q', Q_0', B', D')$ with $Q=Q'$, $Q_0=Q_0'$ and
\Eq{
b'_{ij} &= \case{-b_{ij}&\mbox{if $i=k$ or $j=k$},\\ b_{ij}+\frac{b_{ik}|b_{kj}|+|b_{ik}|b_{kj}}{2}&\mbox{otherwise},}\\
d'_{i}&=d_i.
}

The cluster mutation in direction $k$ induces an isomorphism $\mu_k^q:\bT_q^{\bQ'}\to \bT_q^{\bQ}$ called the \emph{quantum cluster mutation}, defined by
\Eq{
\mu_k^q(X_i')=\case{X_k\inv&\mbox{if $i=k$},\\ \dis X_i\prod_{r=1}^{|b_{ki}|}(1+q_i^{2r-1}X_k)&\mbox{if $i\neq k$ and $b_{ki}\leq 0$},\\\dis X_i\prod_{r=1}^{b_{ki}}(1+q_i^{2r-1}X_k\inv)\inv&\mbox{if $i\neq k$ and $b_{ki}\geq 0$},}
}
where we denote by $X_i'$ the quantum cluster variables of $\cX_q^{\bQ'}$.
\end{Def}

Finally we recall the notion of \emph{amalgamation} of two quantum torus algebras \cite{FG2}.
\begin{Def}\label{amal}
Let $\bQ, \bQ'$ be two cluster seeds and let $\cX_q^{\bQ}, \cX_q^{\bQ'}$ be the corresponding quantum torus algebras. Let $S\subset Q_0$ and $S'\subset Q_0'$ be subsets of the frozen vertices with a bijection $\phi:S\to S'$ such that $d_{\phi(i)}'=d_i$ for $i\in S$. Then the \emph{amalgamation} of $\cX_q^{\bQ}$ and $\cX_q^{\bQ'}$ along $\phi$ is identified with the subalgebra 
\Eq{
\til{\cX}_q\sub\cX_q^{\bQ}\ox \cX_q^{\bQ'}} generated by the variables $\{\til{X}_i\}_{i\in Q\cup Q'}$ where 
\Eq{
\til{X}_i&:= \case{
X_i\ox 1&\mbox{if $i\in Q\setminus S$},\\
1\ox X_i'&\mbox{if $i\in Q'\setminus S'$},
}\\\nonumber
\til{X}_i= \til{X}_{\phi(i)} &:= X_i\ox X_{\phi(i)}'\tab i\in S.
}
\end{Def}

Equivalently, the amalgamation of the corresponding quivers $\bQ,\bQ'$ is a new quiver 
\Eq{
\til{\bQ}:=\bQ *_\phi \bQ'} constructed by gluing the frozen vertices $S$ along $\phi$, defrozening those vertices that are glued, and removing any resulting 2-cycles. We obviously have
\Eq{
\til{\cX}_q\simeq \cX_q^{\til{\bQ}}.}
\section{Positive representations and cluster realization of $\cU_q(\g)$}\label{sec:pos}
The original \emph{positive representations} of $\cU_q(\g_\R)$ is constructed in \cite{FI, Ip2, Ip3} where the generators of the quantum group are represented by positive essentilly self-adjoint operators on the Hilbert space $L^2(\R^N)$ where $N=l(w_0)$. The representations are constructed by certain quantization of the regular representation of $G$ on the functions of the flag variety $\cB:=B_-\setminus G$ (where we used the right cosets) by multiplication on the right.

\begin{Rem}
In later sections, we sometimes call this the \emph{maximal} positive representations, referring to the codimension of $B_-\subset G$, to distinguish it from the more general \emph{parabolic} positive representations. The maximal positive representation is a special case of the parabolic positive representation where we take the parabolic subgroup to be $P_J=B_-$, i.e. we take the subset $J\subset I$ to be the empty set $J=\emptyset$.
\end{Rem}

\subsection{Construction and properties of the positive representations}\label{sec:pos:pos}
Although the positive representations can be canonically defined using an embedding of $\fD_q(\g)$ into a quantum torus algebra \cite{GS, Ip7} (see Section \ref{sec:pos:qemb} for a summary), it is instructive to recall the original motivation of the construction in order to draw parallel with the construction of the parabolic positive representations later.

Recall that for $G$ a split real simple Lie group, the positive representations are originally constructed by certain quantization of the principal series representations of $G$ following the steps below:
\begin{enumerate}
\item First, take the regular representation of $G_{>0}$ acting on $C^\oo(B_{>0}^-\setminus G_{>0})$ by right multiplication, together with the action of a character $\chi_\l$ on the maximal split torus parametrized by $\l\in P_{\R_{>0}}$ in the $\R_{>0}$-span of the dominant weights lattice.
\item Obtain the infinitesimal action of $\g$ as differential operators on $C^\oo(B_{>0}^-\setminus G_{>0})$ using the Lusztig coordinates \eqref{w0coord} of the totally positive part.
\item Apply the formal Mellin transformation and convert the differential operators into finite difference operators, acting on the same coordinates.
\item Finally, we perform a \emph{twisted quantization} to define positive essentially self-adjoint operators acting on $L^2(B_{>0}^-\setminus G_{>0})\simeq L^2(\R^N)$ where $N=l(w_0)$. These operators represent the split real quantum group $\cU_q(\g_\R)$.
\end{enumerate}
Here the twisted quantization is done by first quantizing the operators, and then introduce an \emph{analytic continuation} by rescaling and shifting the variables with certain multiples of $\sqrt{-1}\frac{b+b\inv}{2}$. In particular, due to the involvement of $b\inv$, the positive representations do not have the usual classical limit, even though one can talk about its semi-classical limit by taking log to the finite difference operators.
\begin{Rem}
As remarked in \cite{GS}, the classical principal series representation can naturally be viewed as action on the larger space
$$L^2(U_-\setminus G)$$
for some unipotent radical $U_-\subset B_-$, together with a Cartan action $H$ on the left, such that the classical principal series representations are the irreducible components of the decomposition of the $H$-action by its central characters.
\Eq{
L^2(U_-\setminus G) = \int_{P_{\R_{>0}}}^\o+ \cP_\l d\l.
}
This decomposition naturally lifts to the quantum case and gives a natural identification of the Hilbert space $\cH$ underlying the positive representations of $\cU_q(\g_\R)$.
\end{Rem}

Let $\{u_k, p_k\}_{k=1,...,N}$ be the standard position and momentum operators acting on the Hilbert space $L^2(\R^N, du_1...du_N)$ with the same coordinates.

We summarize the results of the construction as follows:
\begin{Thm}\cite{FI, Ip2, Ip3}\label{thmpos} Let $\bi_0=(i_1,...,i_N)$ be a fixed reduced word of $w_0$. Then there exists a family of irreducible representations $\pi_\l^{\bi_0}$ of $\cU_q(\g_\R)$ on 
\Eq{\cP_{\l}^{\bi_0}\simeq L^2(\R^N, du_1...du_N)} parametrized by $\l\in P_{\R_{>0}}$, or equivalently by $\l=(\l_1,...,\l_n)\in \R_{>0}^{rank(\g)}$, such that 
\begin{itemize}
\item The generators $\be_i,\bf_i,\bK_i$ are represented by positive essentially self-adjoint operators acting on $L^2(\R^{N})$.
\item For any reduced words $\bi_0$ and $\bi_0'$ of $w_0$, we have unitary equivalence
\Eq{\cP_\l\:= \cP_\l^{\bi_0}\simeq \cP_\l^{\bi_0'}}
where each Coxeter move of the reduced words induces a unitary transformation (by quantum cluster mutations via the quantum dilogarithm function).
\item The generators are explicitly given as follows\footnote{Compared with \cite{Ip3} in the non-simply-laced case, we rescaled the variables $u_k$ by $\sqrt{d_{i_k}}$.}:
\Eq{
\pi_\l^{\bi_0}(\bK_i)&=e\left(-2\l_i-d_i\sum_{j=1}^N a_{i, i_j}u_j\right),\label{KK}\\
\pi_\l^{\bi_0}(\bf_i)&=\pi(f_i^-)+\pi(f_i^+):=\sum_{k:i_k=i}\pi(f^{k,-})+\sum_{k:i_k=i}\pi(f^{k,+}),\label{FF}
}
where
\Eq{\pi(f^{k,\pm}) := e\left(\pm\left(d_{i_k}\sum_{j=1}^{k-1} a_{i_k,i_j} u_j +d_{i_k} u_k+2\l_{i_k}\right)+2p_k\right). \label{FF2}}
\item The $\bE_i$ generators corresponding to the right most index $i_N$ of $\bi_0$ is given explicitly by
\Eq{\label{EE} 
\pi_\l^{\bi_0}(\be_{i_N})=\pi(e_{i_N}^+)+\pi(e_{i_N}^-):=e(d_{i_N}u_N-2p_N)+e(-d_{i_N}u_N-2p_N),
}
while for the other generators with $i\neq i_N$ we have in general
\Eq{
\pi_\l^{\bi_0}(\be_i)=\pi(e_i^+)+\pi(e_i^-),\label{EE2}
}
where $\pi(e_i^\pm)$ are obtained from $e(\pm d_{i_N}u_N-2p_N)$ by a sequence of quantum cluster mutations.
\end{itemize}
\end{Thm}
\subsection{Basic quivers}\label{sec:pos:basicq}
Following the suggestion of \cite{SS1}, in \cite{Ip7} we gave a cluster realization of the positive representations, where we construct an explicit embedding of the Drinfeld's double $\fD_q(\g)$ into the quantum torus algebra $\cX_q^{\bD}$ of certain quiver $\bD$. The polarization of $\cX_q^{\bD}$ then recover the positive representations up to unitary equivalence.

In this subsection, we recall the construction of $\bD$, which is the double of the so-called \emph{basic quiver} \cite{Ip7}. Here we follow (and modify) the extended approach described in \cite{GS} which systematically includes the extra vertices $\{e_i^0\}$ that were originally added manually in \cite{Ip7}.

\begin{Def} Let $i,k\in I$. The \emph{elementary quiver} $\over{\bJ}_k(i)$ consists of
\begin{itemize}
\item The vertex set
\Eq{Q=Q_0=(I\setminus\{i\})\cup \{i_l\}\cup \{i_r\}\cup \{k_e\}} 
which are all frozen;
\item The multipliers $D=(d_j)_{j\in Q}$ which is the pull-back of the multipliers \eqref{di} from $I$ under the natural projection $Q\to I$ sending $\{i_l,i_r\}$ to $i$ and $k_e$ to $k$; and 
\item The adjancy matrix $C=(c_{ij})$ which is defined to be
\Eq{
c_{i_l, j}&=c_{j, i_r}=\frac{d_ia_{ij}}{2},\tab j\in I\setminus\{i\},\\
c_{i_l,i_r}&=c_{i_r,k_e}=c_{k_e,i_l}=1.
}
\end{itemize}
The vertices are organized in \emph{levels}, such that the vertex $j\in I\setminus\{i\}$ is placed at level $j$, $\{i_l,i_r\}$ are placed on the left and right of level $i$, and $k_e$ is placed on an extra level labeled by $k'$. We have dashed arrows between the vertex $j$ and $\{i_l, i_r\}$.

Intuitively, we call the set $(I\setminus\{i\})\cup \{i_l\}$ the \emph{left frozen vertices}, and $(I\setminus\{i\})\cup \{i_r\}$ the \emph{right frozen vertices}.
\end{Def}
\begin{Def} We define $\bJ(i)$ to be the full subquiver of $\over{\bJ}_k(i)$ obtained by removing the vertex $\{k_e\}$.
\end{Def}
\begin{Def}\label{auxq} Let $\bi=(i_1,...,i_m)$ be a reduced word. Let
\Eq{\b_j:=s_{i_m}s_{i_{m-1}}\cdots s_{i_{j+1}}(\a_{i_j}),\tab \a_i\in \D_+, j=1,...,m\label{posroots}}
be a chain of positive roots. 

We define the \emph{auxiliary quiver} $\bH(\bi)$ to have frozen vertex set $I$, labelled by $\{i_e\}_{i\in I}$ and placed on level $i'$, and the same multipliers \eqref{di}. The adjancy matrix $C=(c_{ij})$ is given by
\Eq{
c_{ij}:=s_{ij}\frac{d_ia_{ij}}{2},
}
where
\Eq{
s_{ij}:=\case{\mathrm{sgn}(r-s)&\b_s=\a_i\mbox{ and } \b_r = \a_j,\\0&\mbox{ otherwise.}}
}
\end{Def}

\begin{Def}Let $\bi=(i_1,...,i_m)$ be a reduced word. The \emph{basic quiver} $\bQ(\bi)$ is constructed by amalgamating the elementary quivers
\Eq{
\bQ(\bi):= \bJ_\bi^\#(i_1)*\bJ_\bi^\#(i_2)*\cdots*\bJ_\bi^\#(i_m)*\bH(\bi),
}
where
\Eq{
\bJ_\bi^\#(i_j):=\case{\over{\bJ}_k(i_j)& \mbox{if }\b_j=\a_k,\\\bJ(i_j)&\mbox{otherwise,}
}
}
and successively the right frozen vertices of $\bJ_\bi^\#(i_{k-1})$ are amalgamated to the left frozen vertices of $\bJ_\bi^\#(i_{k})$ on the same level. The vertices of $\bH(\bi)$ are amalgamated to the corresponding extra nodes $\{k_e\}$ of $\over{\bJ}_k(i_j)$.

Finally without loss of generality we remove all the vertices that are disjoint from the quiver. We redefine $Q_0$ so that the frozen vertices in the resulting quiver consists of the left- and right-most vertices of each level, as well as the extra vertices $\{i_e\}$.
\end{Def}

\begin{Prop} The quiver $\bQ(\bi_0)$ coincides with the basic quiver described in \cite{Ip7} for a specific choice of a reduced word $\bi_0$ of $w_0$.
\end{Prop}
\begin{proof} The elementary part $\bJ(\bi)$ (the full subquiver without the extra vertices $\{k_e\}$) is identical to the one used in \cite{Ip7}. The arrows concerning the extra vertices $\{k_e\}$ is governed by the construction of positive representations, and is equivalent to how the Coxeter moves of reduced words modify the chain of positive roots \eqref{posroots}.
\end{proof}

\begin{Def}\label{double}
The \emph{symplectic double} of $\bQ(\bi)$ is defined by amalgamating the basic quiver corresponding to the opposite words along the frozen vertices on the left side of $\bQ(\bi)$ and the right side of $\bQ(\bi^{op})$, namely
\Eq{
\bD(\bi) := \bQ(\bi^{op})*\bQ(\bi).
}
\end{Def}

\begin{Rem}
In \cite{Ip7} we call $\bQ(\bi^{op})$ the \emph{mirror quiver}. It is obtained from $\bQ(\bi)$ by a horizontal mirror reflection, followed by reversing all the arrows.
\end{Rem}
\begin{Nota}
Following \cite{Ip7} (mainly for typesetting purpose\footnote{In \cite{GS}, the quiver $\bQ(\bi)$ here is denoted by $\over{\bJ}(\over{\bi})$ and the corresponding labeling is given by $f_i^j:=\mat{i\\j}$ and $e_i^0:=\mat{i\\-\oo}$ instead.}), we label the vertices successively from left to right on the same level by 
\Eq{
\{f_i^j \: i\in I, j=0,...,n_i\}
} where $i$ is the level of the vertex, and $n_i$ is the number of occurrences of $i$ in the reduced word $\bi$. Note that the labels $f_i^j$ (except the right-most one) are naturally ordered according to the reduced word $\bi$.

The extra vertices $\{i_e\}$ are labeled by $\{e_i^0\}_{i\in I}$.

We label the vertices of the quiver $\bQ(\bi^{op})$ by 
\Eq{
\{f_i^{-j}\},\mbox{ and }\{e_i^0\},} and write 
\Eq{|f_i^{-j}|:=f_i^j.}
Note that $\bQ(\bi)$ and $\bQ(\bi^{op})$ share the same vertices $\{f_i^0\}$ and $\{e_i^0\}$ in the double quiver $\bD(\bi)$.
\end{Nota}

\begin{Ex} Consider $\g=\sl_4$ and let $\bi=(3,2,1)$ be a reduced word. Then only $\b_1:=\a_1$ is a simple root and $\bH(\bi)$ is trivial. Hence the basic quiver is the amalgamation of 
\Eq{
\bQ(\bi)=\bJ(3)*\bJ(2)*\over{\bJ}_1(1),
}
see Figure \ref{basicq}.

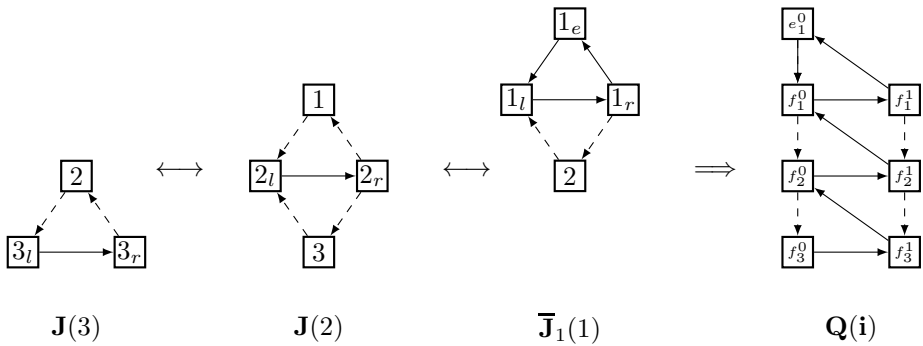
\begin{figure}[H]
\begin{tikzpicture}[baseline=(3), every node/.style={inner sep=0, minimum size=0.4cm, thick, fill=white, draw}, x=0.7cm, y=1cm]
\node (1) at (0,0) {$3_l$};
\node (2) at (2,0) {$3_r$};
\node (3) at (1,1) {$2$};
\drawpath{1,2}{};
\drawpath{2,3,1}{dashed};
\node at (1,-1) [draw=none]{$\bJ(3)$};
\end{tikzpicture}
$\corr$\tab
\begin{tikzpicture}[baseline=(5), every node/.style={inner sep=0, minimum size=0.4cm, thick, fill=white, draw}, x=0.7cm, y=1cm]
\node (4) at (1,1) {$3$};
\node (5) at (0,2) {$2_l$};
\node (6) at (2,2) {$2_r$};
\node (7) at (1,3) {$1$};
\node at (1,0) [draw=none]{$\bJ(2)$};
\drawpath{5,6}{};
\drawpath{6,7,5}{dashed};
\drawpath{6,4,5}{dashed};
\end{tikzpicture}
\tab$\corr$
\begin{tikzpicture}[baseline=(8), every node/.style={inner sep=0, minimum size=0.4cm, thick, fill=white, draw}, x=0.7cm, y=1cm]
\node (8) at (1,3) {$2$};
\node (9) at (0,4) {$1_l$};
\node (10) at (2,4) {$1_r$};
\node (11) at (1,5) {$1_e$};
\node at (1,1) [draw=none]{$\over{\bJ}_1(1)$};
\drawpath{9,10}{};
\drawpath{10,8,9}{dashed};
\drawpath{10,11,9}{};
\end{tikzpicture}
\tab$\Longrightarrow$\tab
\begin{tikzpicture}[baseline=(3), every node/.style={inner sep=0, minimum size=0.4cm, thick, fill=white, draw}, x=0.7cm, y=1cm]
\node (1) at (0,0) {\tiny$f_3^0$};
\node (2) at (2,0) {\tiny$f_3^1$};
\node (3) at (0,1) {\tiny$f_2^0$};
\node (4) at (2,1) {\tiny$f_2^1$};
\node (5) at (0,2) {\tiny$f_1^0$};
\node (6) at (2,2) {\tiny$f_1^1$};
\node (7) at (0,3) {\tiny$e_1^0$};
\drawpath{1,2,3,4,5,6,7,5}{}
\drawpath{7,5,3,1}{dashed}
\drawpath{6,4,2}{dashed}
\node at (1,-1) [draw=none]{$\bQ(\bi)$};
\end{tikzpicture}
\caption{The basic quiver for $\bi=(3,2,1)$.}\label{basicq}
\end{figure}
The resulting double quiver $\bD(\bi)=\bQ(\bi^{op})*\bQ(\bi)$ is shown in Figure \ref{SL4quiver}.
\end{Ex}
Let th verties of $\bQ(\bi)$ be indexed by $Q$, and let those of $\bQ(\bi^{op})$ be indexed by $Q^{op}$. Then there is a standard polarization of the double acting on $L^2(\R^{|\bi|})$ as follows.
\begin{Prop}The \emph{symplectic double polarization} for the quantum torus algebra $\cX_q^{\bD(\bi)}$ is given by
\Eq{
X_i:= \case{
e(\sum_j w_{ij}u_j-2p_i),& i\in Q,\\
e(\sum_j w_{|i|j}u_j+2p_i),& i\in Q^{op},\\
e(2\sum_j w_{ij}u_j),& i\in Q\cap Q^{op},
}
}
where $(w_{ij})_{i,j\in Q}$ is given by \eqref{wij} obtained from the exchange matrix $(b_{ij})_{i,j\in Q}$ of the cluster seed associated to the basic quiver $\bQ(\bi)$.
\end{Prop}

\begin{Cor}\label{rankN} The quiver $\bD(\bi)$ has $2|\bi|+k$ vertices and of rank $2|\bi|$ where $k=|Q\cap Q^{op}|$. In particular, the central characters are generated by $k$ elements.
\end{Cor}
\begin{proof} It suffices to show that the Weyl algebra generators $e(u_i), e(2p_i)$ can be obtained as monomials of the cluster variables. (Equivalently, elements $X_{L}$ for $L\in \L_\R$ in the real span of the defining lattice.) The operator $e(p_i)$ can be obtained easily by taking the ratios of the opposite pair of cluster variables $X_{f_i^j}$ and $X_{f_i^{-j}}$. Therefore we can focus on the position operators. In fact we can focus on the full subquiver of $\bD(\bi)$ by removing the vertices $\{e_i^0\}$. 

If $\bi$ starts with $i$, then the cluster variable $X_{f_i^0}$ consists of a single variable. Hence it follows by induction that one can obtain $e(u_i)$ by successively taking ratios of the cluster variables corresponding to the $i_k$ and $i_{k-1}$-th letter of $\bi$.
\end{proof}
Hence we can introduce the eigenvalues of the central characters $e(\l_i)\in\R, i\in |Q\cap Q^{op}|$ to the polarization of $X_{e_i^0}$. (If we introduce them in the unfrozen variable, then we can do an appropriate unitary transformation $p_i\mapsto \l$ and get rid of them.)
\subsection{Quantum group embedding}\label{sec:pos:qemb}
Finally, we summarize the cluster realization of positive representations using the double quiver $\bD(\bi_0)$ for a longest reduced word $\bi_0$ defined in the previous subsection.

\begin{Thm}\cite{Ip7} There exists an embedding of the Drinfeld's double 
\Eq{\iota: \fD_q(\g)&\inj \cX_q^{\bD(\bi_0)}\nonumber\\
\{\be_i, \bf_i, \bK_i, \bK_i'\} &\mapsto \{e_i, f_i, K_i, K_i'\},\label{dqxq}
}
such that $K_iK_i'$ lies in the center of $\cX_q^{\bD(\bi_0)}$. In particular we have an embedding
\Eq{\iota: \cU_q(\g)&\inj \cX_q^{\bD(\bi_0)}/\<K_iK_i'=1\>. \label{uqxq}}
There exists a polarization $\pi_\l$ of $\cX_q^{\bD(\bi_0)}$ where $\pi_\l(K_iK_i') = 1$ and the other $n$ central characters acting by $e(\l_i) \in\R_{>0}$, such that the composition with the embedding \eqref{uqxq} coincides with the expression of the positive representations $\cP_\l$. 
\end{Thm}
\begin{Def}\label{grouplikepolar} We call $\pi_\l$ the \emph{group-like} polarization of $\cX_q^{\bD(\bi_0)}$.
\end{Def}

One can construct the group-like polarization $\pi_\l$ explicitly by taking successive ratios of the terms $f^{k,\pm}$ in \eqref{FF}. It is then clear that the rank of this polarization is $2N$ (using the same argument as in the proof of Corollary \ref{rankN}), hence by Lemma \ref{polaru} it is unitarily equivalent to the standard polarization of the symplectic double induced by an $Sp(2N)$ action on the lattice.

\begin{Nota}\label{bold} To clarify the formula in the rest of the paper, we will always use \textbf{bold face} to denote elements of the quantum groups $\cU_q$ or $\fD_q$, and the corresponding \textit{Roman script} to denote elements in a quantum torus algebra $\cX_q$.
\end{Nota}

From the explicit construction \eqref{KK}--\eqref{EE} we can rephrase the embedding as follows.
\begin{Cor}\label{embex}\cite{Ip7} The embedding \eqref{dqxq} is such that (cf. Notation \ref{xik}):
\begin{itemize}
\item $f_i$ are represented by a telescoping sum, and $K_i'$ are monomials, explicitly given by
\Eq{
f_i &= \sum_{k=-n_i}^{n_i-1} X_{f_{i^{-n_i}},..., f_i^m}\label{fiX},\\
K_i'&=X_{f_i^{-n_i},..., f_i^{n_i}}\label{KiX}.
}
\item If $\bi_0$ is chosen such that $i_N=i$, then
\Eq{
e_i &= X_{f_i^{n_i}}+X_{f_i^{n_i}, e_i^0},\\
K_i &= X_{f_i^{n_i}, e_i^0, f_i^{-n_i}}.
}
\item $K_i$ are always monomial. 
\item When $\bi_0$ is appropriately chosen, the generators $e_i$ can also be represented by telescopic sums\footnote{With slight modification in type $C_n, F_4$ and $E_8$.}.
\end{itemize}
\end{Cor}

It is known that there is a sequence of cluster transformations that sends $e_i$ or $f_i$ into a single monomial of $\cX_q$ such that the vertices corresponding to the variables are sinks. In particular the embedding \eqref{dqxq} is into the \emph{universally Laurent polynomial} of the cluster algebra generated by the mutation class of $\cX_q$ \cite[Proposition 13.11]{GS}. A conjecture in \cite{Ip7} is that these Laurent polynomials are in fact always polynomials with no negative powers of the variables involved in any cluster, cf. Conjecture \ref{polycon}.

\begin{Rem}\label{KM} In \cite[Proposition 11.2]{GS}, they generalize the same formula of the embedding \eqref{fiX}--\eqref{KiX} to give a homomorphism (not necessarily injective) of the Borel part $\cU_q(\fb_-)$ to $\cX_q^{\bQ(\bi)}$ for \emph{arbitrary} reduced word $\bi$, even in the Kac-Moody case.
\end{Rem}

Finally, combining Theorem \ref{thmpos} and Corollary \ref{embex}, for any root index $i\in I$, we can choose $\bi_0$ with $i_N=i$ again in order to observe that\footnote{The notations used here differ slightly from \eqref{FF}--\eqref{EE2}.}
\begin{Cor}\label{efcom} We have the following decomposition in $\cX_q^{\bQ(\bi^{op})}\ox \cX_q^{\bQ(\bi)}$
\Eq{
e_i &= e_i^++K_i^+e_i^-,\\
f_i &= f_i^-+{K_i'}^-f_i^+,\\
K_i&=K_i^-K_i^+,\\
K_i'&={K_i'}^-{K_i'}^+,
}
where the $(+)$ generators belong to $1\ox \cX_q^{\bQ(\bi)}$ and $(-)$ generators belong to $\cX_q^{\bQ(\bi^{op})}\ox 1$, so that they mutually commute, 
\Eq{
\frac{[e_i^+, f_j^+]}{q_i-q_i\inv} = \d_{ij} {K_i'}^+, \tab \frac{[e_i^-, f_j^-]}{q_i-q_i\inv} = -\d_{ij} K_i^-,
}
and $\{e_i^\pm, f_i^\pm, K_i^\pm, {K_i'}^\pm\}$ satisfy all other quantum groups relations \eqref{KK1}--\eqref{SerreF} (except \eqref{EFFE}).
\end{Cor}

The triple $\{e_i^\pm, f_i^\pm, K_i^\pm\}$ forms the commutation relation of the \emph{Heisenberg double} $\cH_q^\pm(\g)$ \cite{Ka1} of the Borel part of $\cU_q(\g)$. This is an important observation for us to construct a representation of $\fD_q(\g)$ on the double quiver $\bD(\bi)$ by generalizing the Heisenberg double realization (cf. Section \ref{sec:ppr:heid}).

\begin{Ex}\label{a5ex} Consider $\g=\sl_6$ of type $A_5$ and choose the reduced expression $$\bi_0=(1,2,1,3,2,1,4,3,2,1,5,4,3,2,1).$$
The representations of $\be_i$ and $\bf_i$ are given by a telescopic sum of polynomials, represented by \emph{paths} on the quiver\footnote{Formally we should draw the quiver such that the blue arrows are horizontal on their respective level. The quiver we present here first appeared in \cite{SS1} which is more aesthetically pleasing.} presented in Figure \ref{fig-An}. The $e_i$-paths always start with the node $f_i^{n_i}$ and the $f_i$-paths start with $f_i^{-n_i}$. In this notation we have for example
\Eqn{
f_1 &= X_{f_1^{-5}} + X_{f_1^{-5}, f_1^{-4}} + \cdots + X_{f_1^{-5},..., f_1^{4}},\\
K_1'&=X_{f_1^{-5},f_1^{-4},..., f_1^4,f_1^5}.
}
Note that the embedding of the $f_i$ generators \emph{never} include the last node of the $f_i$-paths, which only exists in $K_i'$.

On the other hand, since the word $\bi_0$ ends with $i_N=1$, we have a simple expression forthe embedding of $\be_1$ given by
\Eqn{
e_1=X_{f_1^5}+X_{f_1^5,e_1^0},\\
K_1=X_{f_1^5,e_1^0,f_1^{-5}},
}
again with the same rules as the $f_i$ generators.
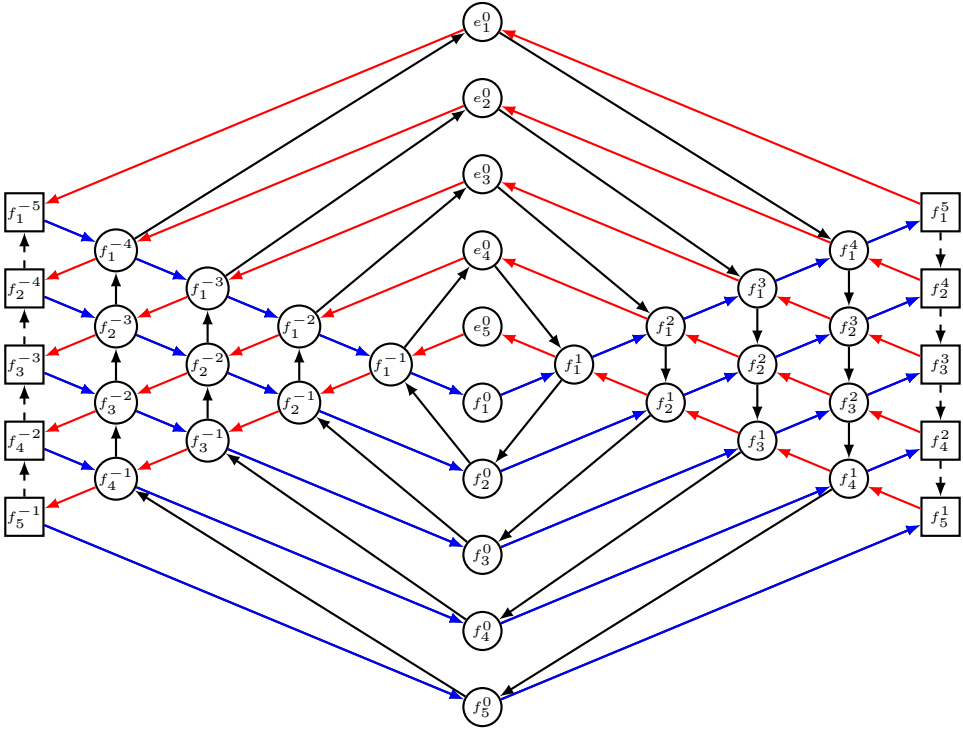
\begin{figure}[!htb]
\centering
\begin{tikzpicture}[every node/.style={inner sep=0, minimum size=0.5cm, thick}, x=1.2cm, y=0.5cm]
\foreach \y [count=\c from 1] in{33,28,21,12,1}
\foreach \x in{1,...,\c}{
	\pgfmathtruncatemacro{\ind}{\x+\y-1}
	\pgfmathtruncatemacro{\indop}{\y+2*\c-\x+1}
	\pgfmathtruncatemacro{\r}{6-\c}
	\pgfmathtruncatemacro{\ex}{\x-1-\c}
	\pgfmathtruncatemacro{\exop}{\c+1-\x}
	\ifthenelse{\x=1}{
		\node(\ind) at (\x-1,2*\c-1-\x)[draw] {\tiny$f_\r^{\ex}$};
		\node(\indop) at (11-\x,2*\c-1-\x)[draw] {\tiny$f_\r^{\exop}$};
	}{
		\node(\ind) at (\x-1,2*\c-1-\x)[draw, circle]{\tiny$f_\r^{\ex}$};
		\node(\indop) at (11-\x,2*\c-1-\x)[draw, circle]{\tiny$f_\r^{\exop}$};
	}
}
\foreach \x[count=\c from 1] in{6,16,24,30,34}{
\node(\x) at (5, 5-\c*2)[draw, circle]{\tiny$f_\c^0$};
}
\foreach \x[count=\c from 1] in{36,37,38,39,40}{
\node(\x) at (5, 15-\c*2)[draw, circle]{\tiny$e_\c^0$};
}

\foreach \x[count=\c from 1]in{33,28,21,12,1}{
\pgfmathtruncatemacro{\ee}{\x+2*\c}{
	\drawpath{\x,...,\ee}{thick}
}
}
\drawpath{11,36,1}{thick, red}
\drawpath{20,10,37,2,12}{thick, red}
\drawpath{27,19,9,38,3,13,21}{thick, red}
\drawpath{32,26,18,8,39,4,14,22,28}{thick, red}
\drawpath{35,31,25,17,7,40,5,15,23,29,33}{thick, red}
\drawpath{34,29,22,13,2,36,10,19,26,31,34}{thick}
\drawpath{3,37,9,18,25,30,23,14,3}{thick}
\drawpath{4,38,8,17,24,15,4}{thick}
\drawpath{5,39,7,16,5}{thick}
\drawpath{1,2,3,4,5,6,7,8,9,10,11}{thick, blue}
\drawpath{12,13,14,15,16,17,18,19,20}{thick,blue}
\drawpath{21,22,23,24,25,26,27}{thick,blue}
\drawpath{28,29,30,31,32}{thick,blue}
\drawpath{33,34,35}{thick,blue}
\drawpath{33,28,21,12,1}{dashed, thick}
\drawpath{11,20,27,32,35}{dashed, thick}
\end{tikzpicture}
\caption{$A_5$-quiver, with the $e_i$-paths colored in red and $f_i$-paths colored in blue.}
\label{fig-An}
\end{figure}
\end{Ex}
\section{Minimal positive representations of $\cU_q(\sl(n+1,\R))$}\label{sec:min}
Following the strategy for the positive representations of $\cU_q(\g_\R)$, given a parabolic subgroup $P_J\subset G$ containing $B_-$, we construct a representation of $\cU_q(\g_\R)$ by quantizing the infinitesimal action of $\g$ on the functions of totally positive partial flag varieties $C^\oo(P_J^{>0}\setminus G_{>0})$ induced by the regular action 
\Eq{
g\cdot f(h) = \chi_\l([hg]_0) f([hg])
}
where $\chi_\l$ is some character on $P_J$ parametrized by the weights $\l\in\R^{|I\setminus J|}$ (cf. Lemma \ref{parachar}).

In this section, we illustrate the above procedure explicitly for the \emph{maximal parabolic subgroup} and construct a family of \emph{positive representations} of $\cU_q(\sl(n+1,\R))$, acting as positive essentially self-adjoint operators on the Hilbert space $L^2(\R^n)$. Since this Hilbert space has the minimal functional dimension possible, we call this family the \emph{minimal positive representations}. The simplicity of such representations is expected to find applications in mathematical physics and integrable systems.

\subsection{Motivational example: $\cU_q(\sl(4,\R))$}\label{sec:min:ex}
Let us consider the standard maximal parabolic subgroup
\Eq{
P_J=\mat{{*}&*&*&0\\{*}&*&*&0\\{*}&*&*&0\\{*}&*&*&*}
}
corresponding to $J=\{1,2\}$. It admits a Langlands decomposition of the form $P_J^{>0}=U_{>0}^-T_{>0}M_{>0}$:
\Eq{
P_J^{>0}=\mat{1&0&0&0\\{*}&1&0&0\\{*}&*&1&0\\{*}&*&*&1}\mat{{*}&0&0&0\\0&*&0&0\\0&0&*&0\\0&0&0&*}\mat{1&*&*&0\\0&1&*&0\\0&0&1&0\\0&0&0&1}
}
The positive characters on $P_J^{>0}$ is parametrized by a single scalar $\l\in \R$ given on the abelian component $T_{>0}$:
\Eq{
h(t_1,t_2,t_3):=h_1(t_1)h_2(t_2)h_3(t_3)}
by
\Eq{
\chi_\l(h(t_1,t_2,t_3))=t_3^{2\l}
}
which only depends on the last coordinate.

Consider the action on $P_J^{>0}\setminus G_{>0}$ by right multiplication, here the right coset can be expressed as
$$\mat{{*}&*&*&0\\{*}&*&*&0\\{*}&*&*&0\\{*}&*&*&*}\mat{1&a&0&0\\0&1&b&0\\0&0&1&c\\0&0&0&1}$$
for $a,b,c\in\R_{>0}$, where in fact we can rewrite the representative in terms of the Lusztig coordinates as
\Eqn{
\mat{1&a&0&0\\0&1&b&0\\0&0&1&c\\0&0&0&1}&=\mat{1&0&0&0\\0&1&0&0\\0&0&1&c\\0&0&0&1}\mat{1&0&0&0\\0&1&b&0\\0&0&1&0\\0&0&0&1}\mat{1&a&0&0\\0&1&0&0\\0&0&1&0\\0&0&0&1}\\
&=x_3(c)x_2(b)x_1(a)
}

The group elements $e^{tX}$, where $X\in\g$, acts by multiplying on the right. More precisely, we only need to consider the action of 
\Eq{
x_i(t)=e^{tE_i},\tab y_i(t)=e^{tF_i}\tab h_i(t)=e^{tH_i}
} 
Hence using the Coxeter moves \eqref{EF}--\eqref{EE2}, we can rearrange the right multiplication in the form
\Eq{
x_3(c)x_2(b)x_1(a)e^{tX} = n\cdot h\cdot x_1(f')x_2(e')x_1(d')x_3(c')x_2(b')x_1(a'), \tab n\in U_-, h\in T.
}

To obtain a regular representation on $C^{\oo}(P_J^{>0}\setminus G_{>0})$, we project onto the coset $P_-\setminus G$ by ignoring the coordinates $d', e', f'$, and apply the character $\chi_\l$ to $h\in T_{>0}$.

For example, 
$$x_3(c)x_2(b)x_1(a)x_1(t)=x_3(c)x_2(b)x_1(a+t),$$
and hence the regular action is given by
$$e^{tE_1}: f(a,b,c)\mapsto f(a+t,b,c).$$
On the other hand
$$x_3(c)x_2(b)x_1(a)y_3(t)=y_3(\frac{t}{1+ct})h_3(1+ct)x_3(\frac{c}{1+ct})x_2(b)x_1(a)$$
and hence the regular action is given by
$$e^{tF_3}: f(a,b,c)\mapsto (1+ct)^{2\l}f(a,b,\frac{c}{1+ct}).$$
We compute the rest of the actions to be

\Eqn{
e^{tE_1}: f(a,b,c)&\mapsto f(a+t,b,c),\\
e^{tE_2}:  f(a,b,c)&\mapsto f(\frac{ab}{b+t},b+t,c),\\
e^{tE_3}:  f(a,b,c)&\mapsto f(a,\frac{bc}{c+t},c+t),\\
e^{tF_1}: f(a,b,c)&\mapsto f(\frac{a}{a t+1}, b(a t+1),c),\\
e^{tF_2}: f(a,b,c)&\mapsto f(a,\frac{b}{bt+1}, c(bt+1)),\\
e^{tF_3}: f(a,b,c)&\mapsto (ct+1)^{2\l} f(a,b, \frac{c}{ct+1}),\\
e^{tH_1}: f(a,b,c)&\mapsto f(ae^{-2t},be^{t},c),\\
e^{tH_2}: f(a,b,c)&\mapsto f(ae^t, be^{-2t}, ce^t),\\
e^{tH_3}: f(a,b,c)&\mapsto e^{2t\l}f(a, be^{t},ce^{-2t}).
}

The infinitesimal action 
\Eq{
X\cdot f := \left.\frac{d}{dt}(e^{tX}\cdot f)\right|_{t=0}
}
can now be computed to be
\Eqn{
E_1&=\del[,a],\\
E_2&=\del[,b]-\frac{a}{b}\del[,a],\\
E_3&=\del[,c]-\frac{b}{c}\del[,b],\\
F_1&=-a^2\del[,a]+ab\del[,b],\\
F_2&=-b^2\del[,b]+bc\del[,c],\\
F_3&=2c\l-c^2\del[,b],\\
H_1&=-2a\del[,a]+b\del[,b],\\
H_2&=a\del[,a]-2b\del[,b]+c\del[,c],\\
H_3&=2\l+b\del[,b]-2c\del[,c].
}

Now applying the \emph{formal Mellin transform} \cite{Ip2}
\Eq{
f(a,b,c)\mapsto \cF(u,v,w):=\int a^ub^vc^w f(a,b,c)dadbdc,} we can turn the differential operators into finite difference operators as follows
\Eqn{
E_1:\cF(u,v,w)&\mapsto (u+1)\cF(u+1,v,w),\\
E_2:\cF(u,v,w)&\mapsto (-u+v+1)\cF(u,v+1,w),\\
E_3:\cF(u,v,w)&\mapsto (-v+w+1)\cF(u,v,w+1),\\
F_1:\cF(u,v,w)&\mapsto (-u+v+1)\cF(u-1,v,w),\\
F_2:\cF(u,v,w)&\mapsto (-v+w+1)\cF(u,v-1,w),\\
F_3:\cF(u,v,w)&\mapsto (2\l-w+1)\cF(u,v,w-1),\\
H_1:\cF(u,v,w)&\mapsto (-2u+v)\cF(u,v,w),\\
H_2:\cF(u,v,w)&\mapsto (u-2v+w)\cF(u,v,w),\\
H_3:\cF(u,v,w)&\mapsto (2\l+v-2w)\cF(u,v,w).
}
\begin{Rem}
By the same argument in \cite{IpGL}, by introducing appropriate measure (a shift with complex parameters), one can identify 
\Eq{
L^2(\R_{>0}^3, d\mu(a,b,c))\simeq L^2(\R^3, dudvdw)
} with the standard Lebesgue measure under the Mellin transformation.
\end{Rem}

Finally, following \cite{Ip2} to perform the twisted quantization, and using Notation \ref{ELnote}, the quantized action of $\cU_q(\sl(4,\R))$ is given (on the rescaled generators) by
\Eqn{
\pi_\l^J(\be_1)&=[u]e(-2p_u),\\
\pi_\l^J(\be_2)&=[-u+v]e(-2p_v),\\
\pi_\l^J(\be_3)&=[-v+w]e(-2p_w),\\
\pi_\l^J(\bf_1)&=[-u+v]e(2p_u),\\
\pi_\l^J(\bf_2)&=[-v+w]e(2p_v),\\
\pi_\l^J(\bf_3)&=[2\l-w]e(2p_w),\\
\pi_\l^J(\bK_1)&=e(-2u+v),\\
\pi_\l^J(\bK_2)&=e(u-2v+w),\\
\pi_\l^J(\bK_3)&=e(v-2w+2\l).
}
One checks explicitly that this gives a representation of $\cU_q(\sl(4,\R))$ as positive essentially self-adjoint operators on $L^2(\R^3)$, in the sense of Remark \ref{domain}.

\begin{Rem}
Recall that the operator $e(2p_u)$ \emph{etc.} acts on the dense subspace $\cW$ by shifting the variable $u\mapsto u-\sqrt{-1}b$, and thus it is the analytic continuation (or \emph{Wick's rotation}) of the classical finite difference operator.
\end{Rem}

Furthermore, following \cite{Ip7}, the above representations can also be represented as the homomorphism from the Drinfeld's double $\fD_q(\g)$ together with a polarization of the quantum torus algebra associated to a quiver with 10 vertices, as shown in Figure \ref{SL4quiver}, which one can verify to be the same as the double quiver $\bD(\bi)$ for the reduced word $\bi=(3,2,1)$.

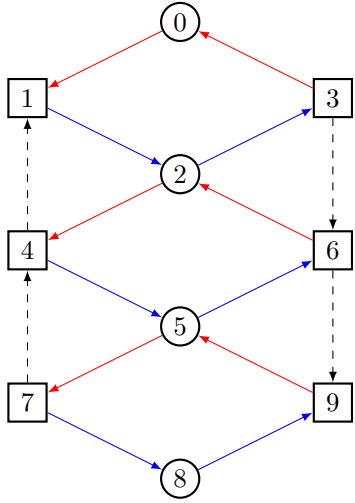
\begin{figure}
\centering
\begin{tikzpicture}[every node/.style={inner sep=0, minimum size=0.5cm, thick, fill=white, draw}, x=2cm, y=1cm]
\node (1) at (0,6) {$1$};
\node (0) at (1,7) [circle]{$0$};
\node (3) at (2,6) {$3$};
\node (4) at (0,4) {$4$};
\node (2) at (1,5) [circle]{$2$};
\node (6) at (2,4) {$6$};
\node (7) at (0,2) {$7$};
\node (5) at (1,3) [circle]{$5$};
\node (9) at (2,2) {$9$};
\node (8) at (1,1) [circle]{$8$};
\drawpath{1,2,3}{blue}
\drawpath{4,5,6}{blue}
\drawpath{7,8,9}{blue}
\drawpath{3,0,1}{red}
\drawpath{6,2,4}{red}
\drawpath{9,5,7}{red}
\drawpath{7,4,1}{dashed}
\drawpath{3,6,9}{dashed}
\end{tikzpicture}
\caption{The double quiver $\bD(\bi)$ for the word $\bi=(3,2,1)$. The labeling is just conveniently assigned. The $e_i$-paths are colored in red, $f_i$-paths in blue.}\label{SL4quiver}
\end{figure}

\Eqn{
e_1&=X_3+X_{3,0},&K_1&=X_{3,0,1},\\
e_2&=X_6+X_{6,2},&K_2&=X_{6,2,4},\\
e_3&=X_9+X_{9,5},&K_3&=X_{9,5,7},\\
f_1&=X_1+X_{1,2},&K_1'&=X_{1,2,3},\\
f_2&=X_4+X_{4,5},&K_2'&=X_{4,5,6},\\
f_3&=X_7+X_{7,8},&K_3'&=X_{7,8,9},
}
where the polarization is explicitly given by
\Eqn{
X_1&=e(-u+v+2p_u),& X_6&=e(-u+v-2p_v),\\
X_2&=e(2u-2v),&X_7&=e(2\l-w+2p_w),\\
X_3&=e(u-2p_u),&X_8&=e(-4\l+2w),\\
X_4&=e(-v+w+2p_v),&X_9&=e(-v+w-2p_w),\\
X_5&=e(2v-2w),&X_{0}&=e(-2u).
}
By Lemma \ref{polaru}, under the unitary transformation of multiplication by $e^{\pi i(u^2+v^2+w^2)}$, which induces the Weil action 
\Eq{
2p_u&\mapsto 2p_u+u, & 2p_v&\mapsto 2p_v+v, & 2p_w&\mapsto 2p_w+w} we recover (up to the parameter $\l$) the standard polarization of the symplectic double. 

In fact the rank of the underlying lattice is easily verified to be $6=10-4$ so it admits a natural polarization on $L^2(\R^3)$. Hence the center of the quantum torus algebra $\cX_q^{\bD(\bi)}$ is generated by 4 elements, namely
\Eqn{
K_1K_1'&=X_{0,1^2,2,3^2},\\
K_2K_2'&=X_{2,4^2,5,6^2},\\
K_3K_3'&=X_{5,7^2,8,9^2},
} and the central character 
\Eq{
C:=X_{0,2,5,8}=e(-4\l).
}

The simplicity of this quiver allows us to generalize the construction easily to the higher rank case in the next subsection.

By obvious symmetry, by choosing another standard parabolic subgroup
\Eq{
P=\mat{{*}&0&0&0\\{*}&*&*&*\\{*}&*&*&*\\{*}&*&*&*}
}
and following the same procedure, we obtain an equivalent representation, where all the arrows of the quiver are reversed.

However, we note that the minimality of the functional dimension of $\cH$ depends on the codimension of the parabolic subgroup chosen, not the fact that $P_J$ is maximal. As we see from the projection $P_J^{>0}\setminus G_{>0}$, the functional dimension is given by the codimension $N-N_J$, and in fact only for type $A_n$ we can achieve the minimal possible functional dimension of $L^2(\R^n)$ by choosing $J=\{1,2,...,n-1\}$ or $\{2,3,...,n\}$ for the maximal parabolic subgroup.

For example, if we choose the standard parabolic subgroup corresponding to the root index $\{1,3\}$:
\Eq{
P=\mat{{*}&*&0&0\\{*}&*&0&0\\{*}&*&*&*\\{*}&*&*&*}
}
we obtain a representation of $\cU_q(\sl(4,\R))$ on $L^2(\R^4)$, represented by the double quiver $\bD(\bi)$ associated to the word $\bi=(2,1,3,2)$ with rank $12-4=8$, see Figure \ref{fig-2132}, where the homomorphism $\fD_q(\g)\to \cX_q^{\bD(\bi)}$ is given by
\Eqn{
\be_1&\mapsto X_3+X_{3,7}+X_{3,7,10}+X_{3,7,10,5},&\bK_1&\mapsto X_{3,7,10,5,1},\\
\be_2&\mapsto X_8+X_{8,0},&\bK_2&\mapsto X_{8,0,4},\\
\be_3&\mapsto X_{11}+X_{11,7}+X_{11,7,2}+X_{11,7,2,5},&\bK_3&\mapsto X_{11,7,2,5,9},\\
\bf_1&\mapsto X_1+X_{1,2},&\bK_1'&\mapsto X_{1,2,3},\\
\bf_2&\mapsto X_4+X_{4,5}+X_{4,5,6}+X_{4,5,6,7},&\bK_2'&\mapsto X_{4,5,6,7,8},\\
\bf_3&\mapsto X_9+X_{9,10},&\bK_3'&\mapsto X_{9,10,11}.
}

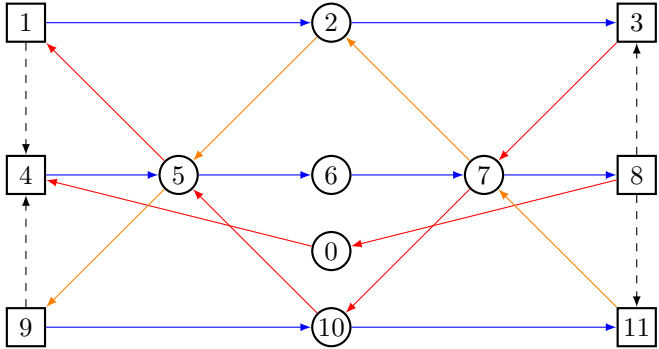
\begin{figure}[h]
\centering
\begin{tikzpicture}[every node/.style={inner sep=0, minimum size=0.5cm, thick, fill=white, draw}, x=2cm, y=1cm]
\node (1) at (0,4) {$1$};
\node (2) at (2,4) [circle]{$2$};
\node (3) at (4,4) {$3$};
\node (4) at (0,2) {$4$};
\node (5) at (1,2) [circle]{$5$};
\node (6) at (2,2) [circle]{$6$};
\node (7) at (3,2) [circle]{$7$};
\node (8) at (4,2) {$8$};
\node (9) at (0,0) {$9$};
\node (10) at (2,0) [circle]{$10$};
\node (11) at (4,0) {$11$};
\node (0) at (2,1) [circle]{$0$};
\drawpath{1,2,3}{blue}
\drawpath{4,5,6,7,8}{blue}
\drawpath{9,10,11}{blue}
\drawpath{3,7,10,5,1}{red}
\drawpath{8,0,4}{red}
\drawpath{11,7,2,5,9}{orange}
\drawpath{1,4}{dashed}
\drawpath{9,4}{dashed}
\drawpath{8,3}{dashed}
\drawpath{8,11}{dashed}
\end{tikzpicture}
\caption{The quiver $\bD(\bi)$ for $\bi=(2,1,3,2)$. The $e_i$-paths are colored in red and orange, $f_i$ paths colored in blue.}\label{fig-2132}
\end{figure}

We will prove in Section \ref{sec:ppr} that for any parabolic subgroup $P_J\subset G$, we can obtain a positive representation for $\cU_q(\g_\R)$, parametrized by $\l\in\R^{|I\setminus J|}$, where the quiver $\bD(\bi)$ is given by a truncated word $\bi$ from the longest word $\bi_0$.

\subsection{Minimal positive representations for $\cU_q(\sl(n+1,\R))$}\label{sec:min:sln}
In fact, from the construction, it is clear that the regular action can be obtained from the maximal positive representation on $L^2(B_{>0}^-\setminus G_{>0})$ by ignoring the Lusztig coordinates \eqref{paraLuscoord} projected out by $P_J$ on the left. In particular all the calculations have been done previously in \cite{Ip2,Ip3}. 

In terms of the polarization, it means that if $u_i$ corresponds to those coordinates, then given the maximal positive representations of $\cU_q(\sl(n+1,\R))$, we obtain the parabolic positive representation by setting
\Eq{
u_i=0\tab \mbox{and}\tab e(\pm p_i)=0.}
Notice that this is in general not a well-defined quotient since the algebraic relation $e^{\pi bp_i}e^{-\pi bp_i}=1$ is violated. However, we will prove in Section \ref{sec:ppr} that in general this procedure allows us to construct arbitrary parabolic positive representations.

\begin{Thm} We have a family of irreducible representations of $\cU_q(\sl(n+1,\R))$ by positive essentially self-adjoint operators on $L^2(\R^n)$, parametrized by $\l\in\R$.
\end{Thm}
\begin{proof} We consider the action on $C^\oo(P_J^{>0}\setminus G_{>0})$ where $J=I\setminus\{n\}$. Following the previous strategy, we obtain the following parabolic positive representations $\cP_\l^J$ given as positive essentially self-adjoint operators on $L^2(\R^n)$, parametrized by $\l\in\R$ as
\Eqn{
\pi_\l^J(\be_1)&=[u_1]e(-2p_1),\\
\pi_\l^J(\be_i)&=[u_i-u_{i-1}]e(-2p_i),& i&\geq 2,\\
\pi_\l^J(\bf_i)&=[-u_i+u_{i+1}]e(2p_i),& i&<n,\\
\pi_\l^J(\bf_n)&=[2\l-u_n]e(2p_n),\\
\pi_\l^J(\bK_i)&=e(-\sum_{j=1}^n a_{ij}u_j),& i&<n,\\
\pi_\l^J(\bK_n)&=e(-\sum_{j=1}^n a_{nj}u_j+2\l)\\
&=e(-2u_n+u_{n-1}+2\l).
}

We can show irreducibility by explicitly reconstructing the Weyl operators $e(u_i)$ and $e(p_i)$ from the representations of $\cU_q(\g_\R)$ as certain rational functions of the generators. Firstly, the action of the $\bK_i$ generator is given by $e(\sum a_{ij}u_j)$ (together with the constant $e(\l)$), hence since the Cartan matrix is invertible, we obtain every $e(u_i)$ from certain (fractional) monomials of $\bK_i$, which is well-defined since $\bK_i$ is positive self-adjoint. Then from the expressions of $\be_i$ we obtain the operators $e(p_i)$ as a rational functions of the generators $\bK_i$ and $\be_i$.

For the cluster algebraic realization, we have a homomorphism
$$\fD_q(\sl_{n+1})\to\cX_q$$ of the Drinfeld's double to the quantum torus algebra $\cX_q:=\cX_q^{\bD(\bi)}$ associated to the double quiver $\bD(\bi)$ where $\bi=(n,n-1,...,3,2,1)$ such that explicitly
\Eqn{
\be_i&\mapsto X_{f_i^1}+X_{f_i^1,f_{i-1}^0},\\
\bf_i&\mapsto  X_{f_i^{-1}}+X_{f_i^{-1},f_i^0},\\
\bK_i&\mapsto  X_{f_i^1,f_{i-1}^0,f_i^{-1}},\\
\bK_i'&\mapsto X_{f_i^{-1},f_i^0,f_i^1}.
}
Here we redefined $f_0^0:=e_1^0$.

\begin{figure}[h]
\centering
\begin{tikzpicture}[every node/.style={inner sep=0, minimum size=0.5cm, thick, fill=white, draw}, x=2cm, y=1cm]
\node (0) at (1,7) [circle]{\tiny$e_1^0$};
\node (1) at (0,6) {\tiny$f_1^{-1}$};
\node (2) at (1,5) [circle]{\tiny$f_1^0$};
\node (3) at (2,6) {\tiny$f_1^1$};
\node (4) at (0,4) {\tiny$f_2^{-1}$};
\node (5) at (1,3) [circle]{\tiny$f_2^0$};
\node (6) at (2,4) {\tiny$f_2^1$};
\node (7) at (0,2) {\tiny$f_3^{-1}$};
\node (8) at (1,1) [circle]{\tiny$f_3^0$};
\node (9) at (2,2) {\tiny$f_3^1$};
\node (10) at (0,0)[draw=none, inner sep=10]{$\vdots$};
\node (11) at (2,0)[draw=none, inner sep=10]{$\vdots$};
\node (12) at (1,0)[draw=none, inner sep=10]{$\vdots$};
\node (13) at (1,-1)[circle]{\tiny$f_{n-1}^0$};
\node (14) at (0,-2){\tiny$f_n^{-1}$};
\node (15) at (1,-3)[circle]{\tiny$f_n^0$};
\node (16) at (2,-2){\tiny$f_n^1$};
\drawpath{1,2,3}{blue}
\drawpath{4,5,6}{blue}
\drawpath{7,8,9}{blue}
\drawpath{14,15,16}{blue}
\drawpath{3,0,1}{red}
\drawpath{6,2,4}{red}
\drawpath{9,5,7}{red}
\drawpath{16,13,14}{red}
\drawpath{14,10,7,4,1}{dashed}
\drawpath{3,6,9,11,16}{dashed}
\end{tikzpicture}
\caption{The quiver $\bD(\bi)$ for $\bi=(n,n-1...,3,2,1)$}
\end{figure}
\end{proof}
\subsection{Non-simple generators}\label{sec:min:nonsimple}
In this subsection, we investigate the non-simple generators stated in Example \ref{nonsimple}.
\begin{Prop}
The non-simple generators can be given explicitly as follows:
\Eq{
\bf_{i,j} &\mapsto X_{f_i^{-1},...,f_{j-1}^{-1},f_i^0,..., f_{j-2}^0}+X_{f_i^{-1},...,f_{j-1}^{-1}, f_i^0,..., f_{j-1}^0},\\
\be_{i,j} &\mapsto X_{f_i^1,...,f_{j-1}^1}+X_{f_i^1,...,f_{j-1}^1, f_{i-1}^0},
}
for $1\leq i<j\leq n$, where again we redefine $f_0^0:=e_1^0$.
\end{Prop}
\begin{proof} We have
\Eqn{
\bf_{1}&\mapsto X_{f_1^{-1}}+X_{f_1^{-1},f_1^0}\\
\bf_{2}&\mapsto X_{f_2^{-1}}+X_{f_2^{-1},f_2^0}\\
\bf_{12}&=T_1\bf_2:=\frac{[\bf_2,\bf_1]_{q^{1/2}}}{q-q\inv}=\frac{q^{1/2}\bf_2\bf_1-q^{-1/2}\bf_1\bf_2}{q-q\inv}&\mapsto X_{f_1^{-1},f_2^{-1},f_1^0}+X_{f_1^{-1},f_1^0,f_2^{-1},f_1^0,f_2^0}\\
}
Noticing that $$\bf_{ij} = T_i\bf_{i+1,j}=\frac{[\bf_{i+1,j},\bf_i]_{q^{1/2}}}{q-q\inv}$$
the expressions follow easily by induction.

The calculations for the expression of $\be_{ij}$ is completely analogous.
\end{proof}

From the explicit expression of the non-simple generators, we notice that the homomorphism $\fD_q(\g)\to \cX_q$ is in general not an embedding. One can solve for $X_i$ in multiple ways by taking ratios of the non-simple generators, hence there are linear dependencies on the PBW basis of $\cU_q(\g)$. In fact, from \cite[Conjecture 2.30]{GS} , the functional dimension of the representations should match only in the case when the parabolic subgroup $P_J$ is the minimal one, i.e. the Borel subgroup $P_J=B_-$.

However, since all the non-simple generators are non-trivial, we can conclude that
\begin{Thm}\label{minR}
The universal $\cR$ operator
\Eq{
\cR = \cK \prod_{\a\in \Phi_+}g_{b_\a}(\be_\a\ox\bf_\a)\label{uniR}
}
is well-defined as unitary operators on tensor products $\cP_\l^J\ox \cP_{\l'}^J$ of parabolic positive representations. 
\end{Thm}
Here $g_b$ is the quantum dilogarithm function and the product is over the natural order of $\Phi_+$ for the standard reduced expression \eqref{standardAw0} of the longest word $\bi_0$, while $\cK$ is the Cartan part of the $\cR$ operator.  See \cite{Ip4} for details.

\subsection{Central characters}\label{sec:min:cc}
In this subsection, we study the central characters of the minimal positive representations in detail. 
\begin{Prop} The rank of $\cX_q$ is $2n$, with center generated by 
$$K_iK_i'=X_{f_i^{-1},f_i^0,f_i^{1},f_i^{1},f_{i-1}^0,f_i^{-1}},\tab i=1,...,n$$
and
$$C:=X_{f_0^0,f_1^0,...,f_n^0}=\prod_{k=0}^n X_{f_k^0}.$$
\end{Prop}
Therefore the parabolic positive representation of $\cU_q(\sl(n+1,\R))$, where we quotient out by the relation $\<K_iK_i'=1\>$, is parametrized by the central element $C$ acting as a scalar
\Eq{
\pi_\l(C)=e(-4\l).
}

Let us consider the \emph{simply-connected} form by adjoining fractional powers of the Cartan generators $\bK_i$ and define
\Eq{\label{adjUq}
\cU_q^{sc}(\g):=\cU_q(\g)[\bK_i^{\pm\frac{1}{h}}]_{i\in I}
 }
where $h$ is the Coxeter number. When $\g=\sl_{n+1}$, $h=n+1$. Positive representations of $\cU_q(\g)$ naturally extends to positive representations of $\cU_q^{sc}(\g)$. 

The center of $\cU_q^{sc}(\g)$ is generated by the Casimirs $\bC_k$ for $k=1,...,\mathrm{rank}(\g)=n$. Since the parabolic positive representation is irreducible, they act on $\cP_\l^J$ as scalars, and we can calculate directly their eigenvalues using the tools of \emph{virtual highest weight vectors} developed in \cite{Ip5}.
Recall that each Casimir $\bC_k$ can be written in the form
$$\bC_k=\bC_K+\bC_{E}$$
where $\bC_K\in \cU_q^{sc}(\fh)$ only depends on (fractional powers of) the Cartan generators $\bK_i$, while $\bC_{E}\in\cU_q^{sc}(\g)$ consists of expressions with $\bE_i$ acting on the right.

Explicitly, if $\cV_k$ is the $k$-th fundamental representation of $\cU_q(\g)$ and $(\vec[W_j])$, $j=1,...,\dim \cV_k=\mat{n+1\\k}$ are the weights of the weight spaces, then \cite{Ip5}
\Eq{
\bC_k = (-1)^n\sum_{j=1}^{\dim \cV_k}q^{\sum_{i=1}^n W_{ji}}  \prod_{i=1}^n \bK_i^{W_{ji}} + \bC_E
}

In particular, the first Casimir, for the $(n+1)$-dimensional standard representation $\cV_1$, has the form
$$\bC_1 = (-1)^n\sum_{j=1}^n q^{n+2-2j}\left( \prod_{k=1}^j\bK_k^{\frac{-2k}{n+1}}\prod_{k=j+1}^n \bK_k^{\frac{2n+2-2k}{n+1}}\right)+\bC_{E}$$
\begin{Ex}\cite{Ip5} When $g=\sl_3$, the positive Casimirs are given by\footnote{This is the adjoint of the ones presented in \cite{Ip5}.}
\Eqn{
\bC_1&:=\bK(q^2\bK_1\bK_2+\bK_1\inv \bK_2+q^{-2}\bK_1\inv \bK_2\inv -q\bK_2\bf_1\be_1-q\inv \bK_1\inv \bf_2\be_2+\bf_{12}\be_{21})\\
\bC_2&:=\bK\inv(q^2\bK_1\bK_2+\bK_1 \bK_2\inv+q^{-2}\bK_1\inv \bK_2\inv-q\inv \bK_2\inv\bf_1\be_1-q \bK_1\bf_2\be_2+\bf_{21}\be_{12} )
} 
where $\bK=\bK_1^{\frac13}\bK_2^{-\frac13}$ and $\be_{ij}=T_i(\be_j)$, $\bf_{ij}=T_i(\bf_j)$ are the non-simple generators (cf. Definition \ref{lusztigbraid}).
\end{Ex}

The idea is that under the (irreducible) parabolic positive representations, these operators acts by certain scalars. Hence their action on test functions on $L^2(\R^n)$ is just multiplication by the same scalars. But when we consider only the algebraic relation, then these operators make sense on generalized distributions $\cF$ also, in the sense that
\Eq{
\<X\cdot \cF, f\> := \<\cF, X\cdot f\>, \tab X\in \cU_q(\g_\R), f\in \cW,
}
where the test functions $f$ are rapidly decreasing and entire. Hence we can find the eigenvalues by acting $\bC_k$ on certain distributions that is annihilated by the $\bE_i$ generators, i.e. these distributions behave like highest weight vector, despite not living in the Hilbert space.

Here we use the complex delta functionals, formally viewed as evaluation at a complex value. We consider the functional
\Eq{
\cF(u_1,...,u_n):=\prod_{k=1}^n\d(u_k-\sqrt{-1}k\c_b)
}
where 
\Eq{
\c_b:=\frac{b-b\inv}{2}.
}
Then by direct calculation we obtain
\Eqn{
\be_i\cdot \cF &= 0,\\
K_i\cdot \cF &=\cF,& i&=1,...,n-1,\\
K_n\cdot \cF &=e(2\l-\sqrt{-1}(n+1)\c_b)\cF.
}
Hence one can interpret this parabolic positive representations as the ``$n$-th fundamental representations" due to the fact that the Cartan generators $H_i$ acts trivially for $i=1,...,n-1$.

Recall that under the maximal positive representations, the spectrum of the (positive) Casimirs $\bC_k$ are positive and lie inside a connected region on the positive quadrant bounded by the \emph{discriminant variety} in $\R^n$ \cite{Ip5}.

\begin{Thm}\label{minCas} The spectrum of the Casimirs $\bC_k$ acting on the parabolic positive representations are real-valued, and lie \emph{outside} the positive spectrum of the positive Casimirs of the maximal positive representations.
\end{Thm}
\begin{proof}
Since the generators $K_i$, $i=1,...,n-1$ acts trivially on the virtual highest weight, the first Casimir acts as the scalar
{\small\Eqn{
\pi_\l^J(\bC_1)&=(-1)^n\sum_{j=1}^n q^{n+2-2j}\pi(\bK_n)^{\frac{2}{n+1}}+(-1)^n q^{-n}\pi(\bK_n)^{\frac{-2n}{n+1}}\\
&=(-1)^n(q^n+\cdots+q^{-n+2})e\left(\frac{2}{n+1}(2\l-i(n+1)\c_b)\right)+(-1)^n q^{-n}e\left(\frac{-2n}{n+1}(2\l-i(n+1)\c_b)\right)\\
&=(-1)^{n+1}(q^{n-1}+q^{n-3}+\cdots+q^{1-n})t+t^{-n}\\
&=(-1)^{n+1}[n]_qt+t^{-n}
}
}
where 
\Eq{
t:=e(\frac{4\l}{n+1})>0
} and we used the fact that $e(inb\inv) = e^{\pi i n}= (-1)^n$.

Recall that the region for the spectrum of the positive Casimirs is bounded by the discriminant of the polynomial
$$P(x)=x^n+C_1(\l)x^{n-1}+\cdots + C_n(\l)x +1$$
where $P(x) = \prod (x+e^{L_i})$ and $L_i$ is the exponential terms of the first Casimir
$$C_1(\l) =\pi(\bC_1)=: \sum e(L_i)$$
In particular, in the parabolic positive representations, we have
$$P(x) = \prod_{k=1}^{n-1}(x+(-1)^{n+1}q^{n-1-2k}t)(x+t^{-n})$$
so that this polynomial always has $(n+1)$ distinct roots, exactly one of which is real, when $t>0$ and $q\neq 1$. Hence the discriminant is either strictly positive or negative when $t$ varies, and lie on one side of the positive spectrum.

Recall that the spectrum of the first positive Casimir has $\pi_\l(\bC_1)\geq n+1$.

We split into cases. Assume $[n]_q>0$. Then when $n$ is even and $t$ is sufficiently large, $\pi_\l^J(\bC_1)$ can take negative values, hence it lies outside of the positive spectrum.

When $n$ is odd, $\pi_\l^J(\bC_1)$ achieve a minimal value of 
$$\frac{(n+1)[n]_q^{\frac{n}{n+1}}}{n^{\frac{n}{n+1}}}<n+1$$ when $\dis t=\left(\frac{n}{[n]_q}\right)^{\frac{1}{n+1}}$ since $[n]_q<n$. Hence again $\pi(\bC_1)$ lies outside the positive spectrum.

The case for $[n]_q<0$ is completely analogue with the parity of $n$ switched.
\end{proof}
As a corollary, we see that the parabolic positive representation behaves slightly differently with respect to the parameter $\l$:
\begin{Cor} The representations $\cP_\l^J$ and $\cP_{-\l}^J$ are \emph{not} unitarily equivalent.
\end{Cor}
\begin{proof} From the explicit expression in the proof above, we see that the spectrum of $\pi_\l^J(\bC_1)$ is not invariant under the exchange $\l\corr -\l$.
\end{proof}
\subsection{Evaluation modules of $\cU_q(\what{\sl}_{n+1})$}\label{sec:min:ev}
Finally, we observe that one can ``wrap around" the minimal quiver in type $A$ to construct a family of positive representations for the affine quantum groups $\cU_q(\what{\sl}_{n+1})$! 

Given the quiver associated to the minimal positive representations of type $A_{n+1}$ above (with $I=\{0,...,n\}$), we form a new quiver $\what{\bD}$ by taking the root index $i\pmod{n+1}$,
\Eq{
f_i^{\e}\mapsto f_{\over{i}}^{\e},\tab \over{i}\in\Z/(n+1)\Z.
}
More precisely, we identify on the original quiver (indexed from $0$ to $n$) the vertex $e_0^0$ with $f_{n}^0$, and adding new dashed arrows between $f_0^{\pm1}$ with $f_{n}^{\pm 1}$, forming a closed loop with $Z_{n+1}$ symmetry. See Figure \ref{affineq}.

\begin{figure}[htb!]
\centering
\begin{tikzpicture}[baseline=(7), every node/.style={inner sep=0, minimum size=0.5cm, thick, fill=white, draw}, x=1.5cm, y=0.87cm]
\node (0) at (1,7) [circle]{\tiny$e_0^0$};
\node (1) at (0,6) {\tiny$f_0^{-1}$};
\node (2) at (1,5) [circle]{\tiny$f_0^0$};
\node (3) at (2,6) {\tiny$f_0^1$};
\node (4) at (0,4) {\tiny$f_1^{-1}$};
\node (5) at (1,3) [circle]{\tiny$f_1^0$};
\node (6) at (2,4) {\tiny$f_1^1$};
\node (7) at (0,2) {\tiny$f_2^{-1}$};
\node (8) at (1,1) [circle]{\tiny$f_2^0$};
\node (9) at (2,2) {\tiny$f_2^1$};
\node (10) at (0,0)[draw=none, inner sep=8]{$\vdots$};
\node (11) at (2,0)[draw=none, inner sep=8]{$\vdots$};
\node (12) at (1,0)[draw=none, inner sep=8]{$\vdots$};
\node (13) at (1,-1)[circle]{\tiny$f_{n-1}^0$};
\node (14) at (0,-2){\tiny$f_n^{-1}$};
\node (15) at (1,-3)[circle]{\tiny$f_n^0$};
\node (16) at (2,-2){\tiny$f_n^1$};
\drawpath{1,2,4,5,7,8,9,5,6,2,3,0,1}{}
\drawpath{14,15,16,13,14}{}
\drawpath{14,10,7,4,1}{dashed}
\drawpath{3,6,9,11,16}{dashed}
\path (1,-3.5) edge[dashed, out=240, in=270] (-1,2);
\path (-1,2) edge[dashed, out=90, in=120] (1,7.5);
\end{tikzpicture}
$\Longrightarrow$
\begin{tikzpicture}[baseline=(7), every node/.style={inner sep=0, minimum size=0.5cm, thick, fill=white, draw}, x=1.5cm, y=0.87cm]
\node (1) at (0,6) {\tiny$f_0^{-1}$};
\node (2) at (1,5) [circle]{\tiny$f_0^0$};
\node (3) at (2,6) {\tiny$f_0^1$};
\node (4) at (0,4) {\tiny$f_1^{-1}$};
\node (5) at (1,3) [circle]{\tiny$f_1^0$};
\node (6) at (2,4) {\tiny$f_1^1$};
\node (7) at (0,2) {\tiny$f_2^{-1}$};
\node (8) at (1,1) [circle]{\tiny$f_2^0$};
\node (9) at (2,2) {\tiny$f_2^1$};
\node (10) at (0,0)[draw=none, inner sep=8]{$\vdots$};
\node (11) at (2,0)[draw=none, inner sep=8]{$\vdots$};
\node (12) at (1,0)[draw=none, inner sep=8]{$\vdots$};
\node (13) at (1,-1)[circle]{\tiny$f_{n-1}^0$};
\node (14) at (0,-2){\tiny$f_n^{-1}$};
\node (15) at (1,-3)[circle]{\tiny$f_n^0$};
\node (16) at (2,-2){\tiny$f_n^1$};
\drawpath{3,15}{}
\drawpath{15,1}{}
\drawpath{1,2,4,5,7,8,9,5,6,2,3}{line width=4pt, white}
\drawpath{14,15,16,13,14}{line width=4pt, white}
\drawpath{1,2,4,5,7,8,9,5,6,2,3}{}
\drawpath{14,15,16,13,14}{}
\drawpath{14,10,7,4,1}{dashed}
\drawpath{3,6,9,11,16}{dashed}
\drawpath{1,14}{dashed, bend right}
\drawpath{16,3}{dashed, bend right}
\end{tikzpicture}
\caption{Construction of the quiver $\what{\bD}$.}\label{affineq}
\end{figure}

The new quiver $\what{\bD}$ consists of $3(n+1)$ vertices. Then it is straightforward to see that that
\begin{Prop}\label{affinerep} The assignments
\Eqn{
\be_i&\mapsto X_{f_i^1}+X_{f_i^1,f_{i-1}^0},\\
\bf_i&\mapsto X_{f_i^{-1}}+X_{f_i^{-1},f_i^0},\\
\bK_i&\mapsto X_{f_i^1,f_{i-1}^0,f_i^{-1}},\\
\bK_i'&\mapsto X_{f_i^{-1},f_i^0,f_i^1},
}
with $i\in \Z/(n+1)\Z$ gives a homomorphism of $\fD_q(\what{\sl}_{n+1})$, the Drinfeld's double of (the Borel part of the) affine quantum group, onto the quantum torus algebra $\cX_q^{\what{\bD}}$ associated to the quiver $\what{\bD}$. 

In particular, a polarization of $\cX_q^{\what{\bD}}$ gives a representation of $\cU_q(\what{\sl}(n+1,\R))$ as positive self-adjoint operators on $L^2(\R^n)$.
\end{Prop}

In the special case of $n=1$, we have a degenerate quiver with no dashed arrows between the frozen nodes, see Figure \ref{sl2hat}. The homomorphism defined by
\Eqn{
\be_0&\mapsto X_3+X_{3,5},&\bK_0&\mapsto X_{3,5,1},\\
\be_1&\mapsto X_6+X_{6,2},&\bK_1&\mapsto X_{6,2,4},\\
\bf_0&\mapsto X_1+X_{1,2},&\bK_0'&\mapsto X_{1,2,3},\\
\bf_1&\mapsto X_4+X_{4,5},&\bK_1'&\mapsto X_{4,5,6}
}
satisfies the defining relation for $\cU_q(\what{\sl}_2)$. In particular the Serre relation $(a_{01}=-2)$ holds:
\Eq{
X_i^3X_j-[3]_q X_i^2X_jX_i+[3]_q X_iX_jX_i^2-X_jX_i^3=0,\tab i\neq j
}
where $X=\be, \bf$. Therefore it gives an irreducible representation of $\cU_q(\what{\sl}_2)$ realized as positive operators on $L^2(\R)$!

\begin{figure}[htb!]
\centering
\begin{tikzpicture}[every node/.style={inner sep=0, minimum size=0.5cm, thick, fill=white, draw}, x=1cm, y=1cm]
\node (1) at (0,2) {$1$};
\node (2) at (2,2) [circle]{$2$};
\node (3) at (4,2) {$3$};
\node (4) at (0,0) {$4$};
\node (5) at (2,0) [circle]{$5$};
\node (6) at (4,0) {$6$};
\drawpath{1,2,3,5,1}{}
\drawpath{4,5,6,2,4}{}
\end{tikzpicture}
\caption{The quiver for $\cU_q(\what{\sl}(2,\R))$.}\label{sl2hat}
\end{figure}

We can check directly that in general the quiver $\what{\bD}$ has rank $2n$, with $3(n+1)-2n = n+3$ central characters. The center is generated by $n+1$ elements
\Eq{
K_iK_i',\tab i=0,...,n,
} as well as $3$ central elements given by 
\Eq{
D_\e:=\prod_{i=0}^n X_{f_i^e}, \tab\e=-1,0,1.}
Note that they are not independent:
\Eq{
\prod_{i=0}^nK_iK_i' = \prod_{\e\in\{-1,0,1\}} D_\e^2
}
in order to produce the correct rank of the quiver.

Therefore a choice of polarization of the quantum torus algebra provides the positive representation acting on the same space $\cP_\l^J\simeq L^2(\R^n)$ as the minimal positive representation of $\cU_q(\sl_{n+1})$, parametrized by two characters. 

Recall the simply-connected form $\cU_q^{sc}(\sl_{n+1})$ defined in \eqref{adjUq} which contains the fractional power $\bK_i^{\frac{1}{n+1}}:= q^{\frac{\bH_i}{n+1}}$. In the positive setting, we have the following homomorphism modified from \cite{Ji2}
$$\cU_q(\what{\sl}_{n+1})\to \cU_q^{sc}(\sl_{n+1})$$
sending
\Eqn{
\be_0&\mapsto e(\mu) \over{\bK}\bf_{1,n},&\be_i&\mapsto \be_i,\tab i>0,\\
\bf_0&\mapsto e(-\mu)\over{\bK}\inv \be_{1,n},&\bf_i&\mapsto \bf_i,\tab i>0,\\
\bK_0&\mapsto \prod_{k=1}^n \bK_k\inv, &\bK_i&\mapsto \bK_i,\tab i>0,
}
where $\mu\in\R$ and
\Eq{
\over{\bK}:=\bK_1^{\frac{n-1}{n+1}}\cdots \bK_n^{\frac{1-n}{n+1}}.}

The action of $\bK_i$ naturally extends to the action of $\over{\bK}$ on $\cP_\l^J$, and the \emph{evaluation module} \cite{Ji2}, denoted by $\cP_{\l}^\mu$, obtained by evaluating the above expression on $\cP_\l^J$ is well-defined. Since $\over{\bK}$ commutes with $\be_{1,n}$ and $\bf_{1,n}$, the evaluation module $\cP_{\l}^\mu$ is realized by positive operators on $L^2(\R^n)$.
\begin{Rem}
In the original evaluation module \cite{Ji2} on the $(n+1)$-dimensional fundamental representation $\pi$ of $\cU_q(\sl_{n+1})$, the Cartan factor is defined to be
\Eq{
\pi(\over{\bK})=q^{E_{11}+E_{n+1,n+1}}}
where $E_{ij}$ is the elementary matrix with $(i,j)$-th entry equals $1$. Recall that under the fundamental representation,
\Eq{\pi(\bH_k)= E_{kk}-E_{k+1,k+1}.}
It is then straightforward to show that 
\Eq{
\sum_{k=1}^n \frac{n+1-2k}{n+1} \pi(\bH_k)+\frac{2}{n+1}I = E_{11}+E_{n+1,n+1}.
}
Letting $\bK_i=q^{\bH_i}$ we obtain the formal expression for $\over{\bK}$ (together with the constant $q^{\frac{2}{n+1}}$ which can be absorbed into $\mu$), which makes sense when evaluating on $\cP_\l^J$.
\end{Rem}

\begin{Thm}
The positive representation $\pi$ of $\cU_q(\what{\sl}_{n+1})$ defined by Proposition \ref{affinerep} is unitarily equivalent to the evaluation module $\cP_\l^\mu$ under the minimal positive representations of $\cU_q(\sl_{n+1})$, where the central characters are determined by the action of the central elements of $\cX_q$:
\Eq{
e(-4\l)&:= \pi(D_0), \\
 e(\mu)&:=\pi(D_0^{\frac{1}{n+1}}D_1).
}
\end{Thm}

\begin{proof} Note that the full subquiver obtained by removing the frozen vertices $\{f_0^{\pm1}\}$ is exactly the same as the minimal quiver $\bD(\bi)$ for $\cU_q(\sl_{n+1})$. Therefore the representation of the subgroup $\<\be_i,\bf_i, \bK_i\>_{i=1,...,n}$ is exactly the minimal positive representations of $\cU_q(\sl_{n+1})$ with central character $\pi(D_0)=e(-4\l)$. Therefore it suffices to show that the action of the $0$-th generators coincides with the required expression of the evaluation homomorphism.

Consider the generator $\bf_0$, which is represented on $\cX_q^{\what{\bD}}$ by\footnote{Recall (cf. Notation \ref{bold}) that we use bold face to denote elements in $\cU_q$, while Roman script to denote elements in $\cX_q$.}
$$f_0 = X_{f_0^{-1}}+X_{f_0^{-1},f_0^0}=X_{f_0^{-1}}(1+qX_{f_0^0}).$$
On the other hand, $\be_{1,n}$ is given explicitly on $\cX_q^{\what{\bD}}$ by
$$e_{1,n} = X_{f_1^1,f_2^1,...,f_n^1}+X_{f_1^1,f_2^1,...,f_n^1,f_0^0}=X_{f_1^1,f_2^1,...,f_n^1}(1+qX_{f_0^0}).$$
Therefore we see that
$$f_0 = X_{f_0^{-1}}X_{f_1^1,f_2^1,...,f_n^1}\inv = X_{f_0^{-1}}X_{f_0^1}D_1\inv e_{1,n},$$
where the two $X$ terms commute. Hence we only need to show that the factor $ X_{f_0^{-1}}X_{f_0^1}$ represents a multiple of $\over{K}$.

In fact, by solving the system of linear equations on the powers of $X_i$, there is a unique solution expressed in terms of $K_i, K_i'$ and $D_\e$ by
$$X_{f_0^{-1},f_0^1}=\prod_{k=1}^n K_k^{-\frac{n+1-k}{n+1}}(K_k')^{-\frac{k}{n+1}}D_0^{\frac{n}{n+1}}D_{-1}D_1.$$

When we specify $K_k' = K_k\inv$ in the quotient, we see that the expression coincides with 
$$\pi(\bf_0)=\pi(D_0^{\frac{n}{n+1}}D_{-1})\pi(\over{\bK}\inv\be_{1,n}).$$
In particular, noting that $\pi(D_{-1}D_0D_1)=1$, the other central character is given by 
$$e(-\mu) = \pi(D_0^{\frac{n}{n+1}}D_{-1})=\pi(D_0^{-\frac{1}{n+1}}D_1\inv)$$ as required. 

By considering $(D_1D_0)\inv e_0$, the calculations for the $\be_0$ generator is completely analogous. Finally, note that in $\cX_q$,
$$D_{-1}D_0D_1 = K_0K_1\cdots K_n = K_0'K_1'\cdots K_n'.$$
Under the positive representations, $\pi(D_{-1}D_0D_1)=1$, hence we obtain
\Eqn{
\pi(\bK_0) &= \pi(\prod_{k=1}^n \bK_k\inv),\\
\pi(\bK_0') &= \pi(\prod_{k=1}^n {\bK_k'}\inv)=\pi(\prod_{k=1}^n \bK_k)
}
as desired.

\end{proof}
\section{Parabolic positive representations}\label{sec:ppr}
In this section, we construct a family of positive representations of $\cU_q(\g_\R)$ by means of quantizing the regular representations over the totally positive partial flag variety $P_J^{>0}\setminus G_{>0}$ for any parabolic subgroup $P_J$, generalizing the construction in the previous section.

\subsection{The Main Theorem}\label{sec:ppr:thm}
Let $\g$ be a simple Lie algebra with root index $I$ and Weyl group $W$. Let $J\subset I$ and let $W_J\subset W$ be the corresponding Weyl subgroup. Let $w_0\in W$ be the longest element, and $w_J\in W_J$ the longest element of $W_J\subset W$, with a choice of reduced words $\bi_0$ and $\bi_J$ respectively. 

\begin{Def} We define the unique element $\over{w}\in W$ such that
\Eq{
w_0 = w_J \over{w}.
}
Let $\over{\bi}$ be a reduced word of $\over{w}$.
\end{Def}

By Lemma \ref{w0head}, we know that 
\Eq{l(\over{w}) = l(w_0)-l(w_J) = N-N_J}
is the codimension of the parabolic subgroup $P_J$ in $G$. Note that we have 
\Eq{
\bQ(\bi_0) = \bQ(\bi_J)*\bQ(\over{\bi})
}
and the symplectic double $\bD(\bi_J)$ (Definition \ref{double}) is naturally a full subquiver of the double $\bD(\bi_0)$, up to the dashed arrows of its frozen vertices.

The Main Theorem of the paper is the following result.
\begin{Thm}\label{main}
Let $\bD(\over{\bi})$ be the symplectic double of the basic quiver $\bQ(\over{\bi})$ associated to the reduced word $\over{\bi}$ of the element $\over{w}$, and let $\cX_q^{\bD(\over{\bi})}$ be the associated quantum torus algebra.

Then there is a homomorphism
$$\fD_q(\g) \to \cX_q^{\bD(\over{\bi})}$$
and the image are universally Laurent polynomials in the cluster mutation class of $\cX_q^{\bD(\over{\bi})}$. 

A polarization of $\cX_q^{\bD(\over{\bi})}$ induces a family of irreducible representations $\cP_\l^J$ of $\cU_q(\g_\R)$ parametrized by the scalars $\l\in\R^{|I\setminus J|}$, acting as positive essentially self-adjoint operators on $L^2(\R^{l(\over{w})})$.
\end{Thm}

\begin{Def}
We call 
\Eq{
\cP_\l^J\simeq L^2(\R^{l(\over{w})})
} the \emph{parabolic positive representations} of $\cU_q(\g_\R)$.
\end{Def}
\begin{Rem} As noted before, when $J=\emptyset$, the parabolic positive representations coincides with the usual (maximal) positive representations constructed previously in \cite{FI, Ip2, Ip3}.
\end{Rem}

The name naturally comes from the following consequence of Theorem \ref{main}.
\begin{Cor}\label{maincor} The parabolic positive representations $\cP_\l^J$ is obtained as certain twisted quantization of the parabolic induction, by ignoring the variables $u_i$ corresponding (under the Mellin transform) to the Lusztig coordinates of $M_{>0}$ of the Langlands decomposition \eqref{langlandsd} of $P_J^{>0}$ in the principal series representations. 
\end{Cor}

More explicitly, a (symplectic double) polarization of $\bD(\over{\bi})$ can naturally be obtained from the group-like polarization (Definition \ref{grouplikepolar}) of $\bD(\bi_0)$ by setting the corresponding Weyl operators $e(u_i)=1$ and $e(\pm p_i)=0$, as well as the parameter $\l_j=0$ for $j\in J$. The non-trivial conclusion of Theorem \ref{main} is that such reduction actually gives a representation of $\cU_q(\g_\R)$.

\begin{Rem} The truncating procedure above is tight. In general we do \emph{not} obtain a representation of $\cU_q(\g_\R)$ by arbitrarily truncating any full subquiver $\bQ(\bi)$ from $\bQ(\bi_0)$ and ignoring the corresponding variables.
\end{Rem}

\subsection{Generalized Heisenberg double}\label{sec:ppr:heid}
Recall that we have the Heisenberg double relations (Corollary \ref{efcom}), which is a special case of the following more general construction:

\begin{Def}\label{defgenH} Let $\{\be_i^\pm,\bf_i^\pm,\bK_i^\pm,{\bK_i'}^\pm\}_{i\in I}$ satisfy the following relations:
\Eq{\label{heirel}
\frac{[\be_i^+, \bf_j^+]}{q_i-q_i\inv}&=\d_{ij}{\bK_i'}^++\w_{ij}\bK_i^+,\\
\frac{[\be_i^-, \bf_j^-]}{q_i-q_i\inv}&=-\d_{ij}{\bK_i}^--\w_{ij}{\bK_i'}^-
}
for some scalars $\w_{ij}\in\C$ and (within each $\pm$) the other standard quantum group relations \eqref{KK1}--\eqref{SerreF} (i.e. except \eqref{EFFE}) of $\cU_q(\g)$.

We call the algebra 
\Eq{
\cH_{q,\w}^{\pm}(\g):=\<\be_i^\pm,\bf_i^\pm,\bK_i^\pm,\bK_i'^\pm\>
} the \emph{generalized Heisenberg double} of $\g$.
\end{Def}

\begin{Rem}
When $\w_{ij}\cong 0$, the above relations are the standard relations for the Heisenberg double \cite{Ka1}.
\end{Rem}

Let $\bi$ be a reduced word. Assume $\<e_i^R,f_i^R,K_i^R,K_i'^R\>$ is a homomorphic image of $\cH_{q,\w}^+(\g)$ in $\cX_q^{\bQ(\bi)}$. Then by symmetry, we have elements $\{e_i^L,f_i^L,K_i^L,{K_i'}^L\}$ in the quantum torus algebra $\cX_q^{\bQ(\bi^{op})}$ obtained by replacing the $X_i$ variables in the $R$ elements with $X_{-i}^{\pm1}$ as follows:
\Eq{
\be_i^L &:= \be_i^R|_{X_i\mapsto X_{-i}\inv},\\
\bf_i^L &:= \be_i^R|_{X_i\mapsto X_{-i}\inv},\\
K_i^L &:= \be_i^R|_{X_i\mapsto X_{-i}},\\
{K_i'}^L &:= \be_i^R|_{X_i\mapsto X_{-i}},
}
where by abuse of notation, we denote by $X_{-i}\in \cX_q^{\bQ(\bi^{op})}$ the corresponding opposite variable of $X_i\in  \cX_q^{\bQ(\bi)}$, i.e. interchanging the labels $\{f_i^{-j}, e_i^0\}\corr \{f_i^j, e_i^0\}$.

\begin{Lem} Define in $\cX_q^{\bQ(\bi^{op})}$ the elements
\Eq{\what{e_i}^L&:=q_i\inv K_i^Le_i^L,\\
\what{f_i}^L&:=q_i\inv {K_i'}^Lf_i^L.
}
Then $\<\what{e_i}^L, \what{f_i}^L, K_i^L, {K_i'}^L\>$ is a homomorphic image of $\cH_{q,\w}^-(\g)$ in $\cX_q^{\bQ(\bi^{op})}$.
\end{Lem}
\begin{proof}By definition and symmetry, we see that
\Eq{
\frac{[e_i^L, f_j^L]}{q_i-q_i\inv}=-\d_{ij}({K_i'}^L)\inv-\w_{ij}(K_i^L)\inv
}
and other quantum group relations are satisfied. Then we have the commutation relation
\Eqn{
\frac{[\what{e_i}^L, \what{f_j}^L]}{q_i-q_i\inv}&=K_i^L{K_j'}^L\frac{[e_i^L,f_j^L]}{q_i-q_i\inv}\\
&=K_i^L{K_j'}^L(-\d_{ij}({K_i'}^L)\inv-\w_{ij}({K_i}^L)\inv)\\
&=-\d_{ij}K_i^L-\w_{ij}{K_j'}^L
}
while all other quantum group relations, including the Serre relation, remains the same. 
\end{proof}

\begin{Prop}\label{genH} The elements in $\cX_q^{\bQ(\bi^{op})}\ox \cX_q^{\bQ(\bi)}$ defined by
\Eq{
e_i&= e_i^R+ \what{e_i}^LK_i^R\\
f_i &=\what{f_i}^L+{K_i'}^Lf_i^R\\
K_i &=K_i^LK_i^R\\
K_i' &={K_i'}^L{K_i'}^R
}
satisfy all the quantum group relations of $\fD_q(\g)$. 

Furthermore, they are elements of the amalgamation $\cX_q^{\bD(\bi)}$ if the elements 
\Eq{K_i, K_i',e_i^R, \what{f_i}^L\in \cX_q^{\bD(\bi)}.\label{genHcon}}
\end{Prop}
Here by abuse of notation, the $L$ generators are elements of $\cX_q^{\bQ(\bi^{op})}\ox1$ and the $R$ generators are elements of $1\ox \cX_q^{\bQ(\bi)}$. 
\begin{proof}
First observe that the $L$ generators and $R$ generators commute by definition. The $q$-commutation relation of $K_i$ with $e_j, f_j$ also follows from definition.

By definition the $L$ and $R$ generators satisfy the Serre relations. Since the expression above can actually be interpreted as a coproduct of the form
\Eqn{
\D(\be) &= 1\ox \be + \what{\be}\ox \bK,\\
\D(\bf) &= \what{\bf}\ox 1+\bK'\ox \bf,
}
so the Serre relations of $e_i$ and $f_i$ follow from the standard algebraic manipulations of the coproduct.

To show the remaining commutation relations, first we observe that $e_i^R$ commutes with $\what{f_j}^L$ and $\what{e_i}^LK_i^R$ commutes with ${K_j'}^Lf_j^R$ for any $i,j \in I$. So it suffices to look at the other cross terms.

We have
\Eqn{
\frac{[e_i^R, {K_j'}^Lf_j^R]}{q_i-q_i\inv}&= {K_j'}^L\frac{[e_i^R, f_j^R]}{q_i-q_i\inv}\\
&={K_j'}^L(\d_{ij}{K_i'}^R+\w_{ij}K_i^R)\\
&=\d_{ij}K_i' + \w_{ij}{K_j'}^LK_i^R,\\
\frac{[\what{e_i}^LK_i^R,\what{f_j}^L]}{q_i-q_i\inv}&=K_i^R\frac{[\what{e_i}^L,\what{f_j}^L]}{q_i-q_i\inv}\\
&=K_i^R(-\d_{ij}K_i^L-\w_{ij}{K_j'}^L)\\
&=-\d_{ij}K_i - \w_{ij}{K_j'}^LK_i^R.
}
Adding together, we obtain the required quantum relations.

The second statement follows from definition. For example, $e_i^R\in \cX_q^{\bD(\bi)}$ is equivalent to the fact that it does not involve variables of the amalgamated vertices, so does $e_i^L$ by definition of symmetry, and hence both belong to $\cX_q^{\bD(\bi)}$.
Furthermore,
$$\what{e_i}^LK_i^R = q_i\inv K_i^L e_i^L K_i^R = q_i\inv K_i e_i^L \in \cX_q^{\bD(\bi)}$$
therefore $e_i\in \cX_q^{\bD(\bi)}$ as required. Similar argument works for the $f_i$ variables.
\end{proof}
\subsection{Decomposition of generators}\label{sec:ppr:dec}
Recall that the (maximal) positive representations can be decomposed into its Heisenberg double counterpart (Corollary \ref{efcom}) as
\Eqn{
e_i &= e_i^++K_i^+e_i^-,\\
f_i &= f_i^-+{K_i'}^-f_i^+,\\
K_i&= K_i^+K_i^-,\\
K_i'&= {K_i'}^+{K_i'}^-,
}
such that the copies $\{e_i^\pm,f_i^\pm, K_i^\pm, {K_i'}^\pm\}$ is an \emph{embedding} of the Heisenberg double $\cH_{q}^\pm(\g):=\cH_{q,0}^\pm(\g)$ into $\cX_q^{\bQ(\bi_0)}$ or $\cX_q^{\bQ(\bi_0^{op})}$. Note that it is consistent with the decomposition of Proposition \ref{genH} for $\w_{ij}\cong0$ if we identify
\Eqn{
e_i^+ &= e_i^R,&e_i^-&=\what{e_i}^L,\\
f_i^+ &= f_i^R,&f_i^-&=\what{f_i}^L,\\
K_i^+ &= K_i^R,&K_i^- &= K_i^L,\\
{K_i'}^+ &= {K_i'}^R,&{K_i'}^-&= {K_i'}^L.
}
\begin{Def} Let $J\subset I$. The \emph{double Dynkin involution} of $i\in I$ is defined to be the unique index $i^{**}\in I$ such that
\Eq{
w_0s_i = s_{i^*} w_0 = s_{i^*}w_J \over{w} = w_J s_{i^{**}}\over{w}.
}
Equivalently, we can compute it using
\Eq{
i^{**}:= (i^{*_{W}})^{*_{W_J}},
}
where we take the Dynkin involution first with respect to the whole group $W$, then with respect to the subgroup $W_J$. Here by convention 
\Eq{
i^{*_{W_J}}:=i,\tab \mbox{if $i\notin J$}.
}
\end{Def}

Now to construct the parabolic positive representations, the main observation is the decomposition of the Heisenberg double generators. Recall that $\bQ(\bi_0) = \bQ(\bi_J)*\bQ(\over{\bi})$.

\begin{Rem}\label{ei0} Due to the rule in Definition \ref{auxq}, we observe that the extra vertices $e_i^0$ of $\bQ(\bi_0)$ corresponds to the labeling $e_{i^{**}}^0$ in $\bQ(\bi_J)$.
\end{Rem}

\begin{Lem}[Decomposition Lemma]\label{decomp} Let $J\subset I$. The embedding 
\Eq{
\cH_q^+(\g)\inj \cX_q^{\bQ(\bi_0)}\subset \cX_q^{\bQ(\bi_J)} \ox \cX_q^{\bQ(\over{\bi})}} can be decomposed into the form
\Eq{
e_i^+ &= \over{e_i} + \over{K_i} e_{i^{**}}^J,\\
f_i^+ &= f_i^J + {K_i'}^J \over{f_i},\\
K_i^+ &= K_{i^{**}}^J\over{K_i},\\
{K_i'}^+ &= {K_i'}^J\over{K_i'},
}
where $e_i^J=f_i^J=0$ and $K_i^J={K_i'}^J=1$ if $i\notin J$, such that
\begin{itemize}
\item $X_i^J\in \cX_q^{\bQ(\bi_J)}\ox1$ and $\over{X_i}\in 1\ox\cX_q^{\bQ(\over{\bi})}$ for $X=e,f,K,K'$, hence they commute with each other.
\item $\{\be_i^J, \bf_i^J, K_i^J, {K_i'}^J\}$ forms a copy of the embedding of $\cH_q^+(\g_J)$ in $\cX_q^{\bQ(\bi_J)}$ where $\g_J$ is the Lie subalgebra of $\g$ corresponding to the root index $J\subset I$.
\item $\over{K_i},\over{K_i'}$ are monomials that $q$-commutes with $\over{e_j},\over{f_j}$ as in \eqref{KK1}--\eqref{KK3}.
\end{itemize}
\end{Lem}
Theorem \ref{main} is now a direct consequence of the following properties:
\begin{Prop} $\<\over{e_i}, \over{f_i}, \over{K_i}, \over{K_i'}\>$ is the image of the generalized Heisenberg double $\cH_{q,\w}^+(\g)$ for some parameters $\w_{ij}$.
\end{Prop}
\begin{proof} By the same argument as in the proof of Proposition \ref{genH}, the generators satisfy the $q$-commutation relations and the Serre relations. It suffices to consider the commutation between $\over{e_i}$ and $\over{f_j}$.

The statement is trivial if either $j\notin J$ or $i^{**}\notin J$. Hence assume $j\in J$ and $i^{**}\in J$. Then same as the proof of Proposition \ref{genH} before, only the cross terms matter.  We have
\Eqn{
\d_{ij}{K_j'}^+&=\frac{[e_i^+, f_j^+]}{q_i-q_i\inv}\\
&=\frac{[\over{e_i}, K_j'^J\over{f_j}]}{q_i-q_i\inv}+\frac{[\over{K_i}e_{i^{**}}^J, f_j^J]}{q_i-q_i\inv}\\
&={K_j'}^J\frac{[\over{e_i}, \over{f_j}]}{q_i-q_i\inv}+\over{K_i}\frac{[e_{i^{**}}^J, f_j^J]}{q_i-q_i\inv}\\
&={K_j'}^J\frac{[\over{e_i},\over{f_j}]}{q_i-q_i\inv}+\d_{i^{**}j}\over{K_i}{K_j'}^J.
}
Therefore
\Eqn{
\frac{[\over{e_i}, \over{f_j}]}{q_i-q_i\inv}=\d_{ij}\over{K_j'}-\d_{i^{**}j}\over{K_i},
}
which is the required relations for the generalized Heisenberg double $\cH_{q,\w}^+$ with 
\Eq{
\w_{ij}:=\case{0&j\notin J\mbox{ or }i^{**}\notin J,\\\d_{i^{**}j}&\mbox{otherwise.}}
}
\end{proof}

\begin{proof}[Proof of Theorem \ref{main}]
Let 
\Eq{
\{e_i^R, f_i^R, K_i^R, {K_i'}^R\}:=\{\over{e_i}, \over{f_i}, \over{K_i}, \over{K_i'}\}.
}
Then by construction they satisfy the condition \eqref{genHcon} of Proposition \ref{genH}, hence we can combine with the opposite copy to obtain a homomorphism of $\fD_q(\g)$ onto $\cX_q^{\bD(\over{\bi})}$.

In particular, under the group-like polarization, truncating the subquiver $\bQ(\bi_J)$ by killing the variables corresponding to the Lusztig coordinates of $M_{>0}$ provides the required polarization for the quantum tous algebra $\cX_q^{\bD(\over{\bi})}$. This amounts to setting $e_i^J, f_i^J\mapsto 0$ and $K_i^J, K_i'^J\mapsto 1$.

As a consequence, if we have an invariant subspace of $\cP_\l^J$, then it naturally induces an invariant subspace of the maximal positive representation $\cP_\l$. Since $\cP_\l$ is irreducible, the parabolic positive repesentation $\cP_\l^J$ is also irreducible.

The fact that the homomorphism sends $\fD_q(\g)$ to the universally Laurent polynomials follows from the proof of the Decomposition Lemma in the next section.
\end{proof}

Pictorially, if the positive representations are represented by the $e_i$ and $f_i$-paths, then the parabolic positive representation is obtained by appropriately contracting the paths. See Section \ref{sec:ex} for explicit examples.

\subsection{Coxeter moves}\label{sec:ppr:cox}
It remains to prove the Decomposition Lemma \ref{decomp}. We need to use the explicit construction of the positive representations, which involve understanding the combinatorics of the Coxeter moves of the reduced words.

Let $w\in W$ with reduced word $\bi=(i_1,...,i_M)$. Let $\cC_{rs}$ denote the Coxeter moves involving position $r<s$ of $\bi$, namely, it is either of the form
\Eq{
(...,\underbrace{i}_{r},\underbrace{j}_{s},....)&\mapsto (..., j,i,....), \label{Cox0}\\
(..., \underbrace{i}_{r},j,\underbrace{i}_{s},....)&\mapsto (..., j,i,j,....), \label{Cox1}\\
(..., \underbrace{i}_{r},j,i,\underbrace{j}_{s},....)&\mapsto (..., j,i,j,i,....), \label{Cox2}
}
where $s=r+1,r+2$ or $r+3$. We do not need to consider type $G_2$.

Recall that a Coxeter move $\bi\mapsto \bi'$ corresponds to a (sequence) of cluster mutations that transforms the quiver
$$\bQ(\bi)\to \bQ(\bi'),$$
and maps the corresponding embedding of $\cU_q(\g)$ generators to the other quantum torus algebra.

\begin{Lem}\label{coxeter}
Let $i,j\in I$. If 
\Eq{
l(s_i w s_j) = l(w),} then there is a sequence of Coxeter moves that brings the reduced word $\bi$ of $w$ which begins with $i$, to a reduced word $\bi'$ which ends with $j$:
\Eq{
\bi=(i,....)\mapsto \bi'=(..., j),
}
such that the sequence of Coxeter moves is of the form
\Eq{
(\cC_{r'_1, s'_1}\to \cdots \to\cC_{r'_{m'}, s'_{m'}})\to (\cC_{r_1, s_1}\to\cdots \to \cC_{r_m, s_m})
}
where 
\Eq{
1=r_1<s_1=r_2<s_2=r_3<\cdots = r_m < s_m = M,
}
i.e. the sequence of Coxeter moves in the second portion consists of a chain of moves which begins with the first letter, increases in indices consecutively, and ends with the last letter.
\end{Lem}
\begin{proof}
We prove this by induction. It is trivial when $l(w)=1,2,3$.

Since by assumption $l(s_iws_j)=l(w)$, we cannot apply Coxeter moves to bring $\bi$ into the word
$$\bi'=(i,...,j)$$
hence at some point, we must need to apply a Coxeter move $\cC_{1,s}$ to change the first index.

After applying this move, the reduced word becomes either of the form
$$(k,\underbrace{i,...}_{\over{\bi}_2})\tab\mbox{or}\tab(k,i,\underbrace{k,...}_{\over{\bi}_3})\tab\mbox{or}\tab (k,i,k,\underbrace{i,...}_{\over{\bi}_4})$$
for some $k\neq i\in I$. It is then easy to see that the element $\over{w}_s$ corresponding to the truncated word $\over{\bi}_s$ satisfies the assumption of the Lemma:
$$l(s_i\over{w}_2s_j)=l(\over{w}_2),\tab l(s_k\over{w}_3s_j)=l(\over{w}_3), \tab l(s_i\over{w}_4s_j)=l(\over{w}_4),$$
because otherwise the assumption on $\bi$ will be violated.
Therefore by induction there exists a sequence of Coxeter moves of the form
$$(\cC_1)\to(\cC_2)$$
that brings $\over{\bi}_s$ to $(..., j)$, where $(\cC_1)$ consists of a sequence of Coxeter moves that does not move the first letter of $\over{\bi}_s$, and $(\cC_2)$ consists of a sequence of Coxeter moves that increases in index consecutively.

Now note that $(\cC_1)$ actually consists of moves of the original word $\bi$ with index $> s$, hence in particular, we can do these moves first, apply $\cC_{1,s}$, and followed by the sequence of moves $(\cC_2)$. The moves $\cC_{1,s}\to(\cC_2)$ gives the second portion of the Coxeter moves that increases in index consecutively as required.
\end{proof}
\subsection{Proof of Decomposition Lemma}
We are now ready to complete the proof of the Decomposition Lemma \ref{decomp} and hence the Main Theorem.
\begin{proof}[Proof of Lemma \ref{decomp}]
For the $K_i'$ and $f_i$ generators it follows from Remark \ref{KM} by the explicit embedding, since the embedding into $\cX_q^{\bQ(\bi_0)}=\cX_q^{\bQ(\bi_J)*\bQ(\over{\bi})}$ obviously can be decomposed. We have
\Eqn{
f_i^+ &= \underbrace{X_{f_i^0}+X_{f_i^0,f_i^1}+\cdots + X_{f_i^0,...,f_i^{n_i^J-1}}}_{f_i^J}+\underbrace{X_{f_i^0,...,f_i^{n_i^J}}+\cdots + X_{f_i^0,...,f_i^{n_i-1}}}_{K_i'^J\over{f_i}},\\
K_i^+ &= \underbrace{X_{f_i^0, f_i^1,..., f_i^{n_i^J}}^J}_{K_i'^J}\underbrace{\over{X}_{f_i^{n_i^J},..., f_i^{n_i}}}_{\over{K_i'}},
}
where $$X_{f_i^{n_i^J}}=X_{f_i^{n_i^J}}^J\ox \over{X}_{f_i^{n_i^J}}$$
corresponds to the cluster variable of the amalgamated notes at level $i$.

Let us now focus on the $e_i^+$ and $K_i^+$ generators. Recall that if the word $\bi_0$ ends with the index $i_N=i$ on the right, then
\Eqn{
e_i^+&=X_{f_i^1},\\
K_i^+&=X_{f_i^1,e_i^0}
}
are simply monomials, and for general word $e_i^+, K_i^+$ are obtained by successive cluster transformations corresponding to the Coxeter moves that brings this $\bi_0$ to the required word.

First, by appropriate cluster transformations on the $\bQ(\bi_J)$ subquiver, we can assume that the $J$ portion of $e_i^+$ is a single cluster variable. Due to the fact that
$$w_0s_i = w_J s_{i^{**}}\over{w},$$
the index of the $J$ portion is given by the double Dynkin involution $i^{**}$.

In the current setup, we need to do the Coxeter moves that bring $\bi_0$ with $i_N=i$ to the form $\bi_0':=(\bi_J,\over{\bi})$ where $\bi_J$ ends with $i^{**}$. By Lemma \ref{coxeter}, the Coxeter moves start by doing a straight decreasing sequence with consecutive indices, and ends at the rightmost letter $i^{**}$ of $\bi_J$. 

By direct application of the quantum cluster mutation formula (cf. Definition \ref{qmut}) and induction, the Coxeter move \eqref{Cox0} amounts to permuting the index, the move \eqref{Cox1} transforms
\Eqn{
e_i^+:X_{i_1}+\cdots + X_{i_1,...,i_j}&\mapsto X_{i_1}+\cdots + X_{i_1,...,i_j,i_{j+1}},\\
K_i^+:X_{i_1,...,i_j, e_i^0}&\mapsto X_{i_1,...,i_j, i_{j+1}, e_i^0}
}
while the move \eqref{Cox2} may transform in two different ways depending on the long and short decorations:
\Eqn{
e_i^+:X_{i_1}+\cdots + X_{i_1,...,i_j}&\mapsto X_{i_1}+\cdots + X_{i_1,...,i_j,i_{j+1}}+ X_{i_1,...,i_j,i_{j+2}},\\
K_i^+: X_{i_1,...,i_j, e_i^0}&\mapsto X_{i_1,...,i_j, i_{j+1}, i_{j+2}, e_i^0},
}
or
\Eqn{
e_i^+: X_{i_1}+\cdots + X_{i_1,...,i_j}&\mapsto X_{i_1}+\cdots +[2]_{q_s}X_{i_1,...,i_j,i_{j+1}}+ X_{i_1,...,i_j,i_{j+1}^2}+ X_{i_1,...,i_j,i_{j+1}^2, i_{j+2}},\\
K_i^+: X_{i_1,...,i_j, e_i^0}&\mapsto X_{i_1,...,i_j, i_{j+1}^2, i_{j+2}, e_i^0},
}
where $[2]_{q_s}:=q^{\frac12}+q^{-\frac12}$. In both cases the right hand side are elements in the quantum torus algebra of the mutated quiver.

Therefore, by induction after the decreasing chain of Coxeter moves, the generators are of the form
\Eqn{
e_i^+ &= \underbrace{X_{i_1}+\cdots + X_{i_1,..., i_{k-1}}}_{\over{e_i}}+X_{i_1,...,i_k},\\
K_i^+ &= X_{i_1,...,i_k, e_i^0}
}
in the final quantum torus algebra $\cX_q^{\bQ(\bi_0')}$, where 
$$X_{i_k}=X_{i_k}^J\ox \over{X}_{i_k}$$
corresponds to the cluster variable of the amalgamated notes at level $i^{**}$,  so that we have the decomposition (also recall Remark \ref{ei0})
\Eqn{
X_{i_1,...,i_k}&=\underbrace{\over{X}_{i_1,...,i_k}}_{\over{K_i}}\underbrace{X_{i_k}^J}_{e_{i^{**}}^J},\\
K_i^+ =X_{i_1,...,i_k, e_i^0}&=\underbrace{\over{X}_{i_1,...,i_k}}_{\over{K_i}}\underbrace{X_{i_k, e_i^0}^J}_{K_{i^{**}}^J}.
}

The rest of the Coxeter moves correspond to cluster mutations that do not involve variables from $\cX_q^{\bQ(\bi_J)}$, and hence keep the terms $e_{i^{**}}^J$ and $K_{i^{**}}^J$ invariant. Therefore we obtain the required decomposition of the quantum group generators.

From the explicit cluster mutations above, we also see that there exists a quantum cluster mutations involving only variables from $\cX_q^{\bD(\over{\bi})}$ that brings $\over{e_i}$ and $\over{f_i}$ to a single variable having the same adjacency of the quiver of the maximal positive representations. In particular they are all sinks, and hence it follows from \cite[Proposition 13.11]{GS} that these generators are universally Laurent polynomials in the quantum cluster mutation class of $\cX_q^{\bD(\over{\bi})}$.
\end{proof}
\subsection{Remarks on modular double}\label{sec:ppr:mod}
So far we have dealt with the positive representations of the split real quantum group $\cU_q(\g_\R)$, but all the results extend naturally to its \emph{modular double}.

Let $q=e^{\pi i b^2}$, 
\Eq{
b_i:=\sqrt{d_i}b,\tab b_s=\sqrt{d}b,}
where $d=\min_{i\in I}(d_i)$ is the minimum of the multipliers \eqref{di}. Define
\Eq{
q^\vee=e^{\pi i b_s^{-2}}
}
Recall \cite{Ip3, GS} that the modular double is defined to be the algebra
\Eq{
\cU_{qq^\vee}(\g_\R)=\cU_q(\g_\R)\ox\cU_{q^\vee}({}^L\g_\R)
}
where ${}^L\g$ is the Langlands dual of $\g$. 
\begin{Thm}\label{mainmod} The parabolic positive representations is a representation of the modular double in the sense of \cite{FI, Ip2, Ip3}. Namely the generators $\{\be_i^\vee, \bf_i^\vee, \bK_i^\vee\}$ of $\cU_{q^\vee}({}^L\g_\R)$ acts by
\Eq{
\pi_\l(\be_i^\vee) &= \pi_\l(\be_i)^{\frac{1}{b_i^2}}\\
\pi_\l(\bf_i^\vee) &= \pi_\l(\bf_i)^{\frac{1}{b_i^2}}\\
\pi_\l(\bK_i^\vee) &= \pi_\l(\bK_i)^{\frac{1}{b_i^2}}
}
as positive self-adjoint operators on the same space $\cP_\l^J$ of the $\cU_q(\g_\R)$ representation, and they commute weakly with the generators of $\cU_q(\g_\R)$ up to a sign.
\end{Thm}
\begin{proof} Note that the right hand side makes sense via functional calculus since our representations are positive self-adjoint. As in the end of the proof of Lemma \ref{domain}, it follows from the fact that under the unitary transformation given by the a sequence of quantum cluster mutations, the generators in $\cX_q^{\bD(\over{\bi})}$ becomes a single monomial with polarization of the form $\pi(X_i)=e^{2\pi b L_i}$.

Hence the modular double counterpart given by
\Eq{
\pi(X_i^\vee):= \pi(X_i)^{\frac{1}{b_i^2}} = e^{2\pi bb_i^{-2}L_i} = e^{2\pi b\inv d_i\inv L_i}
}
provides the Langlands dual polarization of $\cX_{q^\vee}^{\bD(\over{\bi})}$ in the sense that the multipliers $d_i$ are inverted, or equivalently the long and short root decorations are interchanged.
\end{proof}

\section{Examples}\label{sec:ex}
\subsection{Type $E_6$}\label{sec:ex:e6}
We begin by illustrating the construction of parabolic positive representation using type $E_6$ as an example, since it simultaneously captures the subquivers of type $A_n$ and $D_n$, as well as a nontrivial double Dynkin involution $i^{**}$.

In \cite{Ip7}, using the labeling of the Dynkin diagram
$$
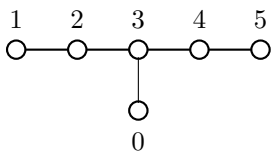
\begin{tikzpicture}[scale=.4]
    \draw[xshift=0 cm,thick] (0 cm, 0) circle (.3 cm);
    \foreach \x in {1,...,4}
    \draw[xshift=\x cm,thick] (\x cm,0) circle (.3cm);
    \foreach \y in {0.15,...,3.15}
    \draw[xshift=\y cm,thick] (\y cm,0) -- +(1.4 cm,0);
    \foreach \z in {1,...,5}
    \node at (2*\z-2,1) {$\z$};
\draw[xshift=0 cm,thick] (4 cm, -2) circle (.3 cm);
  \draw[xshift=0 cm] (4 cm,-0.25) -- +(0 cm,-1.5);
\node at (4,-3){$0$};
  \end{tikzpicture}$$
the positive representations of type $E_6$ corresponding to the reduced word 
\Eq{
\bi_0= (3\;43\;034\; 230432\;12340321\;5432103243054321)
}
of $w_0\in W$, which comes from the embedding of Dynkin diagram 
\Eq{\label{E6chain}
A_1\subset A_2\subset A_3\subset D_4\subset D_5\subset E_6
} is given in Figure \ref{fig-E6}. Recall that the quiver is nothing but the double $\bD(\bi_0)=\bQ(\bi_0^{op})*\bQ(\bi_0)$. 

The embedding of the $\bf_i$ generators are just the horizontal path at level $i$ given as telescoping sums:
\Eq{
\bf_i=X_{i_1}+X_{i_1,i_2}+\cdots + X_{i_1,...,i_{n_i-1}}.
}
On the other hand, the embedding of the $\be_i$ generators are represented by paths of different colors from right to left, which again represent telescoping sums. (Recall that the telescoping sum does not include the last term of the paths, cf. Example \ref{a5ex}.) Finally, the generators $K_i$ and $K_i'$ are represented by monomials of the nodes along the $e_i$ and $f_i$-paths respectively.

We shade the quivers to indicate the boundary of amalgamation of the full subquivers $\bQ(\bi_J)$ for the parabolic subgroups corresponding to the Dynkin chain \eqref{E6chain} above, i.e. $J$ are the different subsets
\Eq{
\{3\}\subset\{3,4\}\subset\{3,4,0\}\subset\{3,4,0,2\}\subset\{3,4,0,2,1\}.
}
(Note: the node $\{f_0^0\}$ is not part of the green region corresponding to the $A_2$ subquiver.)

We can verify then that the $e_i$-paths pass through the correct index of the full subquivers. For example, the $e_2$-path passes through the $0$-th index of the $A_3$ subquiver generated by the root index $J=\{3,4,0\}$, where $(2^{*_W})^{*_{W_J}}=4^{*_{W_J}}=0$.

The double Dynkin involution actually partially explains the behavior of the $e_i$-paths. It was mysterious to us previously the reason why the path goes up and down across the whole quiver passing through different level. We can now see from the quiver diagram that in this case, in fact the $e_i$-paths are forced to take the unique paths along the arrows (without tracing backward) that allow them to go through the correct levels as depicted by the double Dynkin involution.

\begin{sidewaysfigure}
\centering
\begin{tikzpicture}[every node/.style={inner sep=0, minimum size=0.2cm, thick}, x=0.18cm, y=0.47cm]
\xdef\c{0}
\foreach \y[count=\d from 0] in {9,15,23,15,5}{
	\foreach \x in {1,..., \y}{
		\pgfmathtruncatemacro{\ind}{\x+\c}
		\pgfmathsetmacro{\xx}{100*(\x-1)/(\y-1)}
		{	\ifthenelse{\x=1 \OR \x=\y}
		{
		\node[coordinate](\ind) at (\xx, 20-5*\d){};
		}
		{
		\node[coordinate](\ind) at (\xx, 20-5*\d){};
		}
		}
	}
	\xdef\c{\c+\y}
}
\node[coordinate] (68) at (-6,7.5){};
\node[coordinate] (69) at (10,7.5){};
\node[coordinate] (70) at (18,7.5){};
\node[coordinate] (71) at (27,7.5){};
\node[coordinate] (72) at (35,7.5){};
\node[coordinate] (73) at (50,7.5){};
\node[coordinate] (74) at (65,7.5){};
\node[coordinate] (75) at (73,7.5){};
\node[coordinate] (76) at (82,7.5){};
\node[coordinate] (77) at (90,7.5){};
\node[coordinate] (78) at (106,7.5){};
\foreach \y [count=\d from 79] in {24,13.33,11.67,17.5,22,2.5}{
\node[coordinate](\d) at (50, \y){};
}
\fill[blue!10] (83)--(3)--(13)--(29)--(70)--(71)--(29)--(51)--(59)--(43)--(75)--(76)--(43)--(21)--(7)--cycle;
\fill[yellow!10] (82)--(15)--(31)--(71)--(72)--(31)--(52)--(53)--(84)--(57)--(58)--(41)--(74)--(75)--(41)--(19)--cycle;
\fill[red!10] (80)--(34)--(33)--(72)--(73)--(33)--(53)--(84)--(57)--(39)--(73)--(74)--(39)--(38)--cycle;
\fill[green!10] (34)--(54)--(55)--(56)--(38)--(80)--cycle;
\fill[purple!10] (35)--(36)--(37)--(81)--cycle;
\xdef\c{0}
\foreach \y[count=\d from 0] in {9,15,23,15,5}{
	\foreach \x in {1,..., \y}{
		\pgfmathtruncatemacro{\ind}{\x+\c}
		\pgfmathsetmacro{\xx}{100*(\x-1)/(\y-1)}
		\ifthenelse{\d=5}{
			\ifthenelse{\x=1 \OR \x=4 \OR \x=8 \OR \x=11}{
				\node(\ind) at (\x-6+\xx, 7.5)[draw]{};
				}
				{
				\ifthenelse{\x=3 \OR \x=9}{
				\node(\ind) at (0.5*\x-3+\xx, 7.5)[draw,circle]{\ind};
				}
				{
				\node(\ind) at (5*\x-30+\xx, 7.5)[draw, circle]{\ind};
				}
				}
			}
		{	\ifthenelse{\x=1 \OR \x=\y}
		{
		\node(\ind) at (\xx, 20-5*\d)[draw]{};
		}
		{
		\node(\ind) at (\xx, 20-5*\d)[draw, circle]{};
		}
		}
	}
	\xdef\c{\c+\y}
}
\node (68) at (-6,7.5)[draw]{};
\node (69) at (10,7.5)[draw, circle]{};
\node (70) at (18,7.5)[draw, circle]{};
\node (71) at (27,7.5)[draw, circle]{};
\node (72) at (35,7.5)[draw, circle]{};
\node (73) at (50,7.5)[draw, circle]{};
\node (74) at (65,7.5)[draw, circle]{};
\node (75) at (73,7.5)[draw, circle]{};
\node (76) at (82,7.5)[draw, circle]{};
\node (77) at (90,7.5)[draw, circle]{};
\node (78) at (106,7.5)[draw]{};

\foreach \y [count=\d from 79] in {24,13.33,11.67,17.5,22,2.5}{
\node(\d) at (50, \y)[draw, circle]{};
}
\drawpath{63,48,25,10,1}{dashed,}
\drawpath{68,25}{dashed,}
\drawpath{9,24,47,62,67}{dashed,}
\drawpath{47,78}{dashed, }
\xdef\cc{1}
\foreach \y in {9,24,47,62,67,78}{
	\drawpath{\cc,...,\y}	{}
	\pgfmathtruncatemacro{\ind}{\y+1}
	\xdef\cc{\ind}
}
\drawpath{2,79,8}{}
\drawpath{4,83,6}{}
\drawpath{65,51,29,13,3}{}
\drawpath{7,21,43,59,65}{}
\drawpath{19,5,15,16,82,18}{}
\drawpath{57,39,17,33,53}{}
\drawpath{38,73,34}{}
\drawpath{54,84,56}{}
\drawpath{35,80,37,55,35}{}
\drawpath{9,79,1}{purple,}
\drawpath{24,8,22,44,76,42,58,40,74,38,80,34,72,32,52,30,70,28,12,2,10}{red,}
\drawpath{47,23,45,60,43,20,41,18,39,56,37,81,35,54,33,16,31,14,29,50,27,11,25}{orange,}
\drawpath{62,46,77,44,21,6,19,82,15,4,13,28,69,26,48}{green,}
\drawpath{67,61,45,22,7,83,3,12,27,49,63}{blue,}
\drawpath{78,46,61,66,59,42,75,40,57,84,53,32,71,30,51,64,49,26,68}{brown,}
\foreach \y[count=\c from 1] in {4,7,11,7,2}{
\node at (-2.2, 25-5*\c){\tiny$\bf_\c$};
\node at (102, 25-5*\c){\tiny$\be_\c$};
}
\node at (-6, 8.2){\tiny$\bf_0$};
\node at (106, 8.2){\tiny$\be_0$};
\node at (50,6.9){\tiny$f_0^0$};
\foreach \y[count=\c from 0] in {2.5, 24.3,13.33,11.37,17.5,22.3}{
\node at (52, \y){\tiny$e_\c^0$};
}
\end{tikzpicture}
\caption{The $E_6$ quiver, with the $e_i$-paths (from right to left) in different colors.}
\label{fig-E6}
\end{sidewaysfigure}
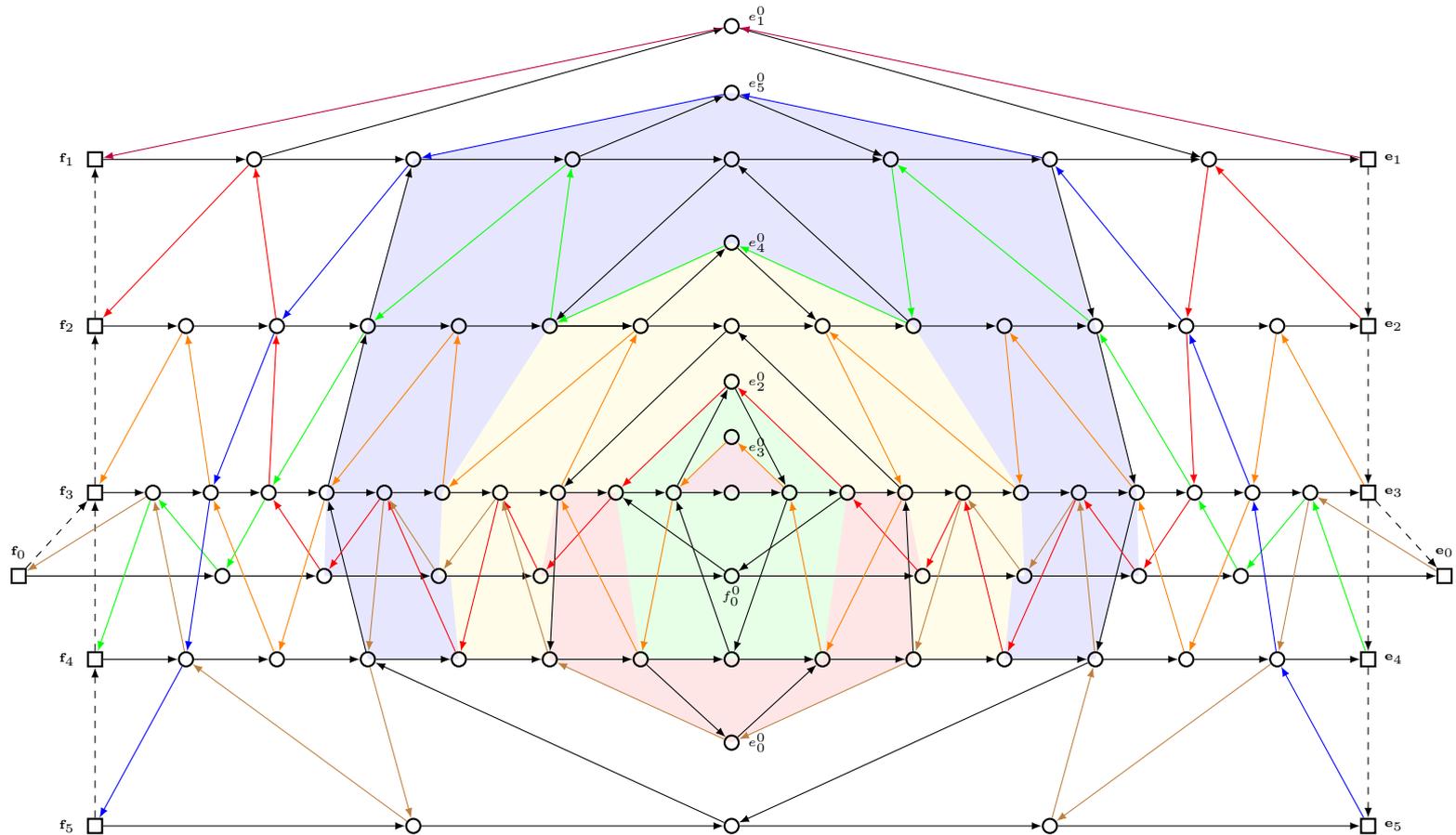

As another example, by considering the parabolic subgroup corresponds to the $D_5$ subquiver corresponding to $J=\{0,1,2,3,4\}$ (i.e. the whole colored region), we truncate the quiver and obtain the $e_i$ and $f_i$-paths representing the parabolic positive representations on the quiver $\bD(\over{\bi})$ as indicated in Figure \ref{fig-E6para}. This is a ``minimal" representation in the sense that the $P_J$ has the smallest codimension in $G$.

\begin{figure}[!htb]
\centering
\begin{tikzpicture}[every node/.style={inner sep=0, minimum size=0.2cm, thick}, x=0.12cm, y=0.47cm]
\xdef\c{0}
\foreach \y[count=\d from 0] in {5,7,9,7,5,5}{
	\foreach \x in {1,..., \y}{
		\pgfmathtruncatemacro{\ind}{\x+\c}
		\pgfmathsetmacro{\xx}{100*(\x-1)/(\y-1)}
		\ifthenelse{\d=5}{
			\ifthenelse{\x=1 \OR \x=5}{
				\node(\ind) at (2.5*\x-7.5+\xx, 7.5)[draw]{};
				}
				{
				\node(\ind) at (\xx, 7.5)[draw, circle]{};
				}
			}
		{	\ifthenelse{\x=1 \OR \x=\y}
		{
		\node(\ind) at (\xx, 20-5*\d)[draw]{};
		}
		{
		\node(\ind) at (\xx, 20-5*\d)[draw, circle]{};
		}
		}
	}
	\xdef\c{\c+\y}
}
\node(39) at (50, 24)[draw, circle]{};
\node at (52, 24.3){\tiny$e_1^0$};
\foreach \y[count=\c from 1] in {4,7,11,7,2}{
\node at (-2.2, 25-5*\c){\tiny$\bf_\c$};
\node at (102, 25-5*\c){\tiny$\be_\c$};
}
\node at (-6, 8.2){\tiny$\bf_0$};
\node at (106, 8.2){\tiny$\be_0$};

\xdef\cc{1}
\foreach \y in {5,12,21,28,33,38}{
	\drawpath{\cc,...,\y}	{}
	\pgfmathtruncatemacro{\ind}{\y+1}
	\xdef\cc{\ind}
}
\drawpath{2,39,4}{}
\drawpath{5,39,1}{purple}
\drawpath{12,4,10,18,36,16,8,2,6}{red}
\drawpath{21,11,19,26,17,24,15,7,13}{orange}
\drawpath{28,20,37,18,9,16,35,14,22}{green}
\drawpath{33,27,19,10,3,8,15,23,29}{blue}
\drawpath{38,20,27,32,25,30,23,14,34}{brown}
\drawpath{29,22,13,6,1}{dashed}
\drawpath{34,13}{dashed}
\drawpath{5,12,21,28,33}{dashed}
\drawpath{21,38}{dashed}
\end{tikzpicture}
\caption{Parabolic positive representations of $D_5\subset E_6$, with the $e_i$-paths (from right to left) in different colors.}
\label{fig-E6para}
\end{figure}
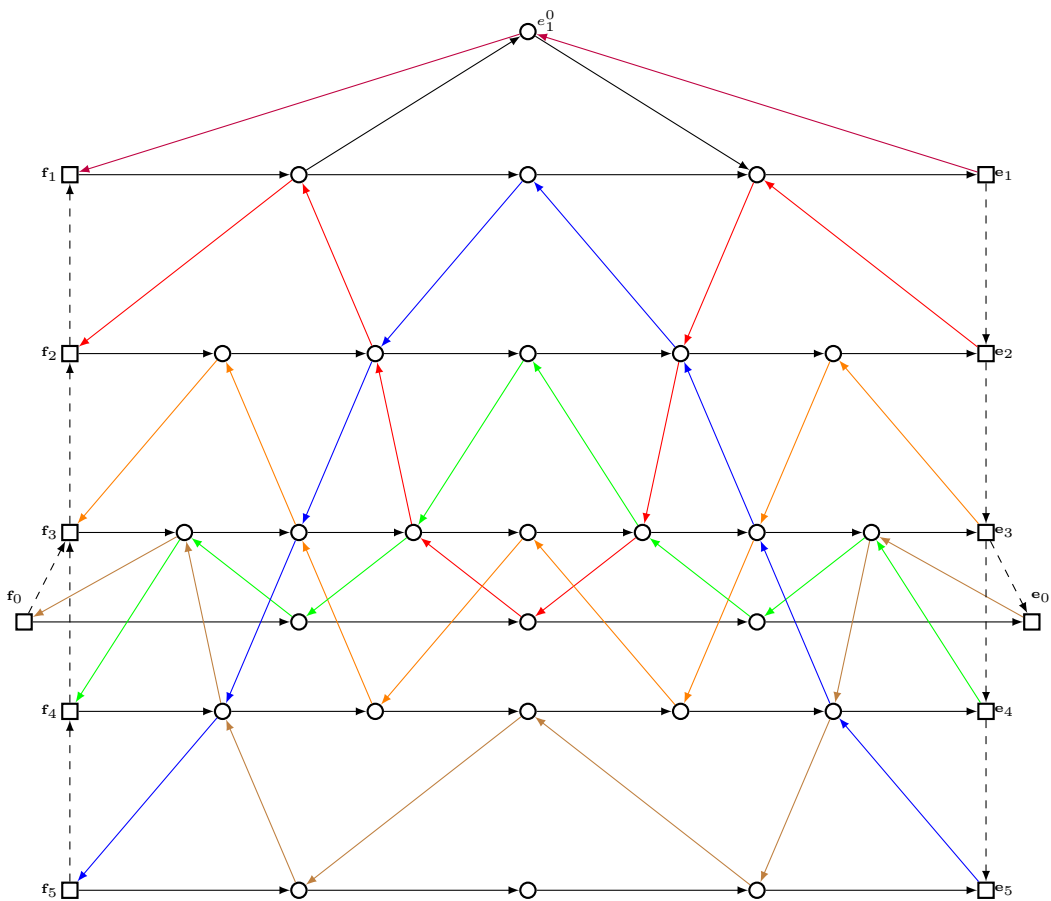

\subsection{Type $B_n$}\label{sec:ex:bn}
The construction obvious works for arbitrary types, including non-simply-laced case. We demonstrate the minimal parabolic positive representations in type $B_n$. Recall that $\{1\}$ is the short root.

When $J=\{1,2,..., n-1\}$ we choose the longest reduced word\footnote{The same word used in \cite{Ip3}.} to be
\Eq{
\bi_0=(1212\;32123\;4321234\;...\; n\;n-1... n-1\; n)
}
so that $\over{\bi}=(n\;n-1\;... 1... \; n-1\; n)$, and we obtain the parabolic positive representations as depicted in Figure \ref{Bn1} on the quiver $\bD(\over{\bi})$, where the $e_i$ and $f_i$-paths are shown in red and blue respectively. (In this quiver, the top row has multiplier $d_i=\frac12$.)

\begin{figure}[htb!]
\centering
\begin{tikzpicture}[every node/.style={inner sep=0, minimum size=0.3cm, thick, fill=white, draw}, x=2cm, y=2cm]
\node (1) at (0,5) {};
\node (2) at (2,5) [circle]{};
\node (3) at (4,5) {};
\node (4) at (0,4) {};
\node (5) at (1,4) [circle]{};
\node (6) at (2,4) [circle]{};
\node (7) at (3,4) [circle]{};
\node (8) at (4,4) {};
\node (9) at (0,3) {};
\node (10) at (1,3)[circle] {};
\node (11) at (2,3)[circle] {};
\node (12) at (3,3)[circle] {};
\node (13) at (4,3) {};
\node (14) at (0,2) {};
\node (15) at (1,2)[circle] {};
\node (16) at (2,2)[circle] {};
\node (17) at (3,2)[circle] {};
\node (18) at (4,2) {};
\node (19) at (2,1)[circle] {};
\node at (-0.2,5)[draw=none]{\small$\bf_1$};
\node at (-0.2,4)[draw=none]{\small$\bf_2$};
\node at (-0.2,3)[draw=none]{\small$\bf_3$};
\node at (-0.2,2)[draw=none]{\small$\bf_4$};
\node at (4.3,5)[draw=none]{\small$\be_1$};
\node at (4.3,4)[draw=none]{\small$\be_2$};
\node at (4.3,3)[draw=none]{\small$\be_3$};
\node at (4.3,2)[draw=none]{\small$\be_4$};
\drawpath{1,2,3}{blue}
\drawpath{4,5,6,7,8}{vthick, blue}
\drawpath{9,10,11,12,13}{vthick, blue}
\drawpath{14,15,16,17,18}{vthick, blue}
\drawpath{3,7,2,5,1}{vthick, red}
\drawpath{8,12,6,10,4}{vthick,red}
\drawpath{13,17,11,15,9}{vthick,red}
\drawpath{18,19,14}{vthick,red}
\drawpath{15,19,17}{vthick}
\drawpath{1,4,9,14}{dashed, vthick}
\drawpath{18,13,8,3}{dashed, vthick}
\end{tikzpicture}
\caption{Parabolic positive representations of type $B_4$ with $J=\{1,2,3\}$. The $e_i$ and $f_i$-paths are shown in red and blue respectively.}\label{Bn1}
\end{figure}
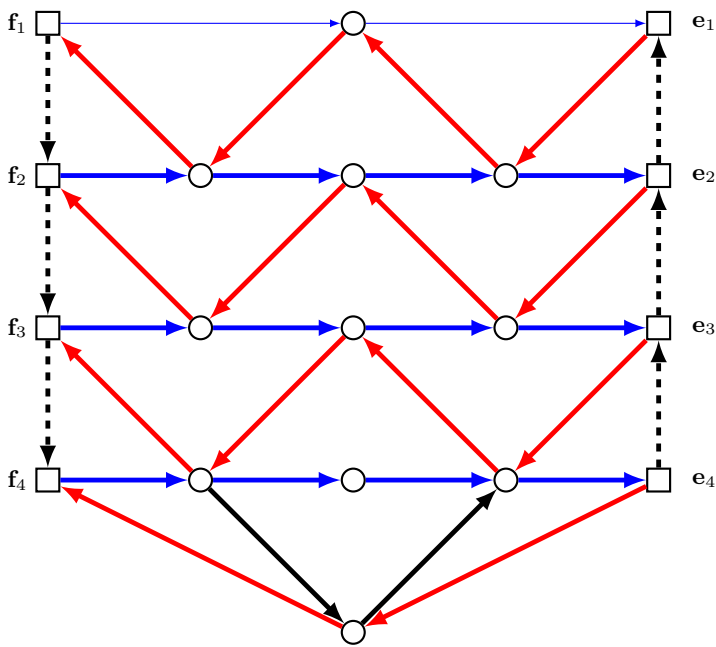

On the other hand, for $J=\{2,3,..., n\}$, $\g_J$ is of type $A_{n-1}$. Then the longest reduced word\footnote{This is the reduced word used in \cite{Le}.} $\bi_0$ is of the form
\Eq{
\bi_0 = \bi_J \bi_{A_n},
}
where $\bi_J$ a longest word for $w_J$, and it turns out $\bi_{A_n}$ is also the standard longest word \eqref{standardAw0}, i.e. we can take
\Eq{
\bi_0 = (n,n-1, n,...,2,3,..., n,\; 1,2,1,3,2,1,...,n,... 1).
}
Then
$\over{\bi}= \bi_{A_n}$ and we obtain the parabolic positive representations on $\bD(\over{\bi})$ as in Figure \ref{Bn2}. However, note that this is \emph{not} the usual type $A_n$ full quiver since we have extra multipliers in this Example. In particular the auxiliary quiver $\bH(\over{\bi})$ only consists of a single node $\{e_1^0\}$.

\begin{figure}[htb!]
\centering
\begin{tikzpicture}[every node/.style={inner sep=0, minimum size=0.3cm, thick, fill=none, draw,circle}, x=0.75cm, y=2cm]
\node (1)[rectangle] at (0,4) {};
\node (2)[red] at (2,4) {};
\node (3)[red] at (4,4) {};
\node (4)[red] at (6,4) {};
\node (5) at (8,4) {};
\node (6)[red] at (10,4) {};
\node (7)[red] at (12,4) {};
\node (8)[red] at (14,4) {};
\node (9)[rectangle] at (16,4) {};
\node (10)[rectangle] at (2,3) {};
\node (11) at (4,3) {};
\node (12) at (6,3) {};
\node (13) at (8,3) {};
\node (14) at (10,3) {};
\node (15) at (12,3) {};
\node (16)[rectangle] at (14,3) {};
\node (17)[rectangle] at (4,2) {};
\node (18) at (6,2) {};
\node (19) at (8,2) {};
\node (20) at (10,2) {};
\node (21)[rectangle] at (12,2) {};
\node (22)[rectangle] at (6,1) {};
\node (23) at (8,1) {};
\node (24)[rectangle] at (10,1) {};
\node (25) at (8,5) {};
\node [draw=none] at (-0.5,4){\small$\bf_1$};
\node [draw=none] at (1.5,3){\small$\bf_2$};
\node [draw=none] at (3.5,2){\small$\bf_3$};
\node [draw=none] at (5.5,1){\small$\bf_4$};
\node [draw=none] at (16.5,4){\small$\be_1$};
\node [draw=none] at (14.5,3){\small$\be_2$};
\node [draw=none] at (12.5,2){\small$\be_3$};
\node [draw=none] at (10.5,1){\small$\be_4$};
\drawpath{1,2,3,4,5,6,7,8,9}{};
\drawpath{2,25,8}{};
\drawpath{10,11,12,13,14,15,16}{vthick};
\drawpath{10,11,12,13,14,15,16}{vthick};
\drawpath{17,18,19,20,21}{vthick};
\drawpath{22,23,24}{vthick};
\drawpath{22,17,10,1}{dashed, vthick};
\drawpath{9,16,21,24}{dashed, vthick};
\drawpath{9,25,1}{green};
\drawpath{16,8,15,20,23,18,11,2,10}{red, vthick}
\drawpath{21,15,7,14,19,12,3,11,17}{orange, vthick}
\drawpath{24,20,14,6,13,4,12,18,22}{blue, vthick}
\end{tikzpicture}
\caption{Parabolic positive representations of type $B_4$ with $J=\{2,3,4\}$.  The $e_i$-paths (from right to left) are shown in different colors.}\label{Bn2}
\end{figure}
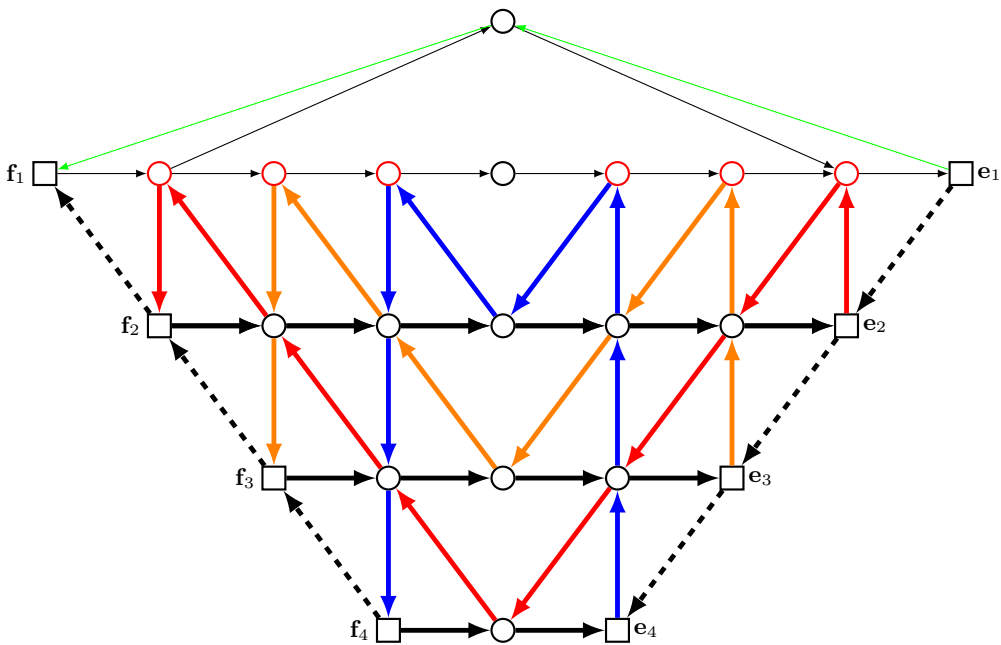
Here the red circles on a node $k$ of the quiver indicates a modification of the $e_i$-paths by doubling the nodes with a $[2]_{q_s}$ factor as in the proof of Lemma \ref{decomp}:
\Eq{
\cdots + X_{..., j} + [2]_{q_s}X_{..., j,k} + X_{..., j,k^2} + X_{..., j,k^2, l}+ \cdots.
}

As remark in Section \ref{sec:ppr:mod}, we still have the modular double transformation
\Eq{
\pi(\be_i^\vee):=\pi(\be_i)^{\frac{1}{b_i^2}}
} to reproduce the parabolic positive representations of $\cU_{q^\vee}(\g_{C_n})$ in type $C_n$, where the quiver remains the same, but with the multipliers of each nodes changes accordingly by $\dis d_i\mapsto \frac{1}{2d_i}$.
\section{Further discussions}\label{sec:open}
There are several natural questions arising from the construction of parabolic positive representations that will be interesting to understand, some of which are motivated from the simplest case of the minimal positive representations in Section \ref{sec:min}. 

During the construction, we have used the basic quiver $\bQ(\bi)$ which is naturally associated to the Poisson structure of the \emph{partial configuration space} $\mathrm{Conf}_u^e(\cA)$ described in \cite{GS}. Therefore one should try to understand and possibly simplify the proofs of the parabolic positive representations by quantizing the geometrical methods using, perhaps, partial decorated $G$-local system, where the decorations are provided by partial flags. However, it seems we do not have a simple association of the partial quiver $\bQ(\bi)$ to the triangles of triangulated surfaces since the frozen degree of each edge does not match.

\subsection*{Polynomial image and Lusztig's braid group action} 

We have seen that the homomorphism
$$\fD_q(\g)\to \cX_q^{\bD(\bi)}$$
lies in the universally Laurent polynomials of the quantum cluster mutation class of $\cX_q^{\bD(\bi)}$. However from all calculations so far we observe that the generators $e_i, f_i$ are in fact always \emph{polynomials}, without any negative powers of $X_k$ involved. 
\begin{Con}\label{polycon} The image of the generators of $\fD_q(\g)$ in $\cX_q^{\bD(\bi)}$ are polynomials in the cluster variables for \emph{any} cluster in the mutation class of the quantum torus algebra.
\end{Con}
This assumption actually allows us to verify the Lusztig's braid group action on $\bD(\bi_0)$. Recall from Definition \ref{lusztigbraid} that we have Lusztig's braid group action $T_i$, which is known to be a cluster map \cite{GS, Le}, i.e. represented by cluster transformation on $\cX_q^{\bD(\bi_0)}$. In particular, the automorphism for the longest word $\bi_0$:
$$T_{\bi_0}:=T_{i_1}\cdots T_{i_N}$$
interchanges the action of the $\be_i$ and $\bf_i$ generators as in Proposition \ref{Ti0}. In terms of quiver, this means that there exists a sequence of mutations such that the basic quiver is mirror reflected and the three frozen sides $\{f_i^0, f_i^{n_i}, e_i^0\}$ are cyclically permuted.

\begin{Con} There exists a cluster map of $\cX_q^{\bD(\over{\bi})}$ for any parabolic positive representations $\cP_\l^J$ that interchanges the action of $\be_i$ and $\bf_i$.
\end{Con}
Using the explicit formula \eqref{fiX}, one can compute directly that the non-simple generators $f_\a$ are always nonzero on the double\footnote{This is not true without the double. For example if we take $\bi=(1,2)$ then $f_{12}=0$ on $\cX_q^{\bQ(\bi)}$.} $\cX_q^{\bD(\bi)}$. Hence a consequence of this conjecture is that
\begin{Cor}\label{Rcon} The universal $\cR$ operator \eqref{uniR} is well-defined on tensor products $\cP_\l^J\ox \cP_{\l'}^{J'}$ of parabolic positive representations as unitary transformation.
\end{Cor}
For the parabolic positive representations, we have verified in a few simple cases that indeed we can interchange the generators by cluster transformations. But in general the mutation sequence for $T_{\bi_0}$ is very complicated and meshed up all the variables, and it is unclear how it can be factorized in the parabolic case. 

For example, Ian Le previously constructed explicitly the mutation sequence for the $T_{\bi_0}$ action in type $E_8$ corresponding to the Coxeter words 
$$\bi_0:=(\underbrace{12345670 \cdots 12345670}_{\mbox{$15$ copies}})^{op}.$$
By construction we know that the generators $\bf_i$ each involve 15 monomials on $\cX_q^{\bQ(\bi_0)}$. Using the mutation sequence, we observe that during the mutations, the $\bf_0$ generator grows up to $825887337$ terms before dropping down to $147249$ terms\footnote{The number of monomial terms for the other generators are $(1,2,3,244,245,246,247)$.} at the end giving the required action of the generator $\be_0$. During the process all the expressions involved are polynomials of the cluster variables as conjectured, and it will be important to give a more combinatorial description of these polynomials.

\subsection*{Casimir operators}

In the minimal positive representations, we have studied in detail the central characters and the action of the Casimir operators, and show that they lie outside the spectrum of the positive Casimirs of the maximal case. It is then natural to propose the following
\begin{Con} The spectrum of the Casimirs of arbitrary parabolic positive representations $\cP_\l^J$ of $\cU_q(\g_\R)$ are all real-valued and disjoint in $\R^n$.
\end{Con}
The techniques using virtual highest weight vector should be generalized to arbitrary parabolic case, and in particular, for the maximal parabolic subgroup corresponding to $J=I\setminus\{j\}$ for a single index $j$, we expect again that the representations behave like the ``fundamental representations'' in the sense that $K_i$ acts trivially for $i\in J$ on the virtual highest weight functionals.

Understanding the spectrum of the Casimirs is an important step towards the decomposition of the tensor product of positive representations. It is clear \cite{Ip7} that the parabolic positive representations of the tensor product
$$\cP_\l^J\ox \cP_{\l'}^{J'}$$
can be constructed by amalgamating two copies $\bD(\bi)*\bD(\bi')$ side by side by concatenating the corresponding $e_i$ and $f_i$-paths. It is then a natural question to decompose it into irreducible components, and together with the existence of the universal $\cR$ operator in Conjecture \ref{Rcon}, possibly provide us with another candidate of \emph{continuous braided tensor category}.

In the case of the usual (maximal) positive representations, the decomposition was done in type $A_n$ in \cite{SS2, SS3} by a sequence of cluster mutations corresponding to flipping the triangulation on a two-punctured disk in order to simplify the Casimir operators and compare them with the Hamiltonians of a Coxeter-Toda conformal field theory. The same trick will not work here as pointed out above, because the quivers involved now are not naturally associated to any triangles of triangulations of punctured Riemann surfaces. We remark that the decomposition in the other types for the usual positive representations is also still an open question.

\subsection*{Modules for affine quantum groups}

Finally, the construction of the evaluation modules of $\cU_q(\what{\sl}_{n+1})$ in Section \ref{sec:min:ev} utilizes the symmetry of the minimal quiver $\bD(\over{\bi})$, and one can try to construct similar quivers for the positive representations of other quantum affine or Kac-Moody group. As pointed out in Remark \ref{KM}, the Borel part of a general Kac-Moody quantum group can be explicitly constructed. For the full quantum group, in type $A_n^{(1)}$ it is clear that one approach is to use the evaluation module to construct a higher rank quiver from those of $\cU_q(\sl_{n+1})$, and it will be interesting to understand the combinatorics behind and find natural generalization to other types. 

\end{document}